\documentclass{amsart}

\usepackage{amsmath, amsthm, amscd, amsfonts, amssymb}

 \newtheorem{thm}{Theorem}[section]
 \newtheorem{cor}[thm]{Corollary}
 \newtheorem{lem}[thm]{Lemma}
 \newtheorem{prop}[thm]{Proposition}
 \theoremstyle{definition}
 
 \theoremstyle{remark}
 \newtheorem{rem}[thm]{Remark}
 \newtheorem{exm}[thm]{Example}
\newcommand{\A}{\mathcal{A}}      
\newcommand{\U}{\mathcal{U}}      
\newcommand{\C}{\mathfrak{S}}
\begin{document}

\nocite{*}

\title{Derivations on semidirect products of Banach algebras}

\author{ H. Farhadi  and  H. Ghahramani }
\thanks{{\scriptsize
\hskip -0.4 true cm \emph{MSC(2010)}:  16E40; 46H40; 46H25.
\newline \emph{Keywords}: semidirect product, Banach algebra, derivation, automatic continuity, first cohomology group.\\}}

\address{Department of
Mathematics, University of Kurdistan, P. O. Box 416, Sanandaj,
Iran.}
\email{h.farhadi@sci.uok.ac.ir}
\email{h.ghahramani@uok.ac.ir; hoger.ghahramani@yahoo.com}

\maketitle
\begin{abstract}
Let $\A$ and $\U$ be Banach algebras such that $\U$ is also a Banach $\A$-bimodule with compatible algebra operations, module actions and norm. By defining an approprite action, we turn $l^1$-direct product $\A\times\U$ into a Banach algebra such that $\A$ is closed subalgebra and $\U$ is a closed ideal of it. This algebra, is in fact semidirect product of $\A$ and $\U$ which we denote it by $\A\ltimes\U$ and every semidirect products of Banach algebras can be represented as this form. In this paper we consider the Banach algebra $\A\ltimes\U$ as mentioned and study the derivations on it. In fact we consider the automatic continuity of the derivations on $\A\ltimes\U$ and obtain some results in this context and study its relation with the automatic continuity of the derivations on $\A$ and $\U$. Also we calculate the first cohomology group of $\A\ltimes\U$ in some different cases and establish relations among the first cohomology group of $\A\ltimes\U$ and those of $\A$ and $\U$. As applications of these contents, we present various results about the automatic continuity of derivations and the first cohomology group of direct products of Banach algebras, module extension Banach algebras and $\theta$-Lau products of Banach algebras.
\end{abstract}
\section{Introduction}
Let $\A$ be a Banach algebra (over the complex field $\mathbb{C}$), and $\U$ be a Banach $\A$-bimodule. A linear map $\delta:\A\rightarrow\U$ is called a \textit{derivation} if $\delta (ab)=a\delta(b)+\delta (a)b$ holds for all $a,b\in\A$. For any $x\in\U$, the map $id_x:\A\rightarrow\U$ given by $id_x(a)=ax-xa$ is a continuous derivation called \textit{inner derivation}. The set of all continuous derivations from $\A$ into $\U$ is denoted by $Z^{1}(\A,\U)$ while $N^{1}(\A,\U)$ denotes the set of all inner derivations from $\A$ into $\U$. $Z^{1}(\A,\U)$ is a linear subspace $\mathbb{B}(\A,\U)$ (the space of all continuous linear maps from $\A$ into $\U$) and $N^{1}(\A,\U)$ is a linear subspace of $Z^{1}(\A,\U)$. We denote by $H^{1}(\A,\U)$, the quotient space $\frac{Z^{1}(\A,\U)}{N^{1}(\A,\U)}$ which is called \textit{the first cohomology group of $\A$ with coefficients in $\U$}. We call the the first cohomology group of $\A$ with coefficients in $\A$ briefly by first cohomology group of $\A$. Derivations and cohomology group are important subjects in the study of Banach algebras. Among the most important problems related to them are these questions; Under what conditions a derivation $\delta:\A\rightarrow\U$ is continuous? Under what conditions one has $H^{1}(\A,\U)=(0)$? (i.e., every continuous derivation from $\A$ into $\U$ is inner). Also the calculation of $H^{1}(\A,\U)$, up to isomorphism of linear spaces, is a considerable problem.
\par 
The problem of continuity of derivations is related to the subject of automatic continuity which is an important subject in mathematical analysis. Many studies have been performed in this regard and it has a long history. We may refer to \cite{da} for more information which is a detailed source in this subject. Here we only review the most important obtained results concerning the automatic continuity of derivations. Johnson and Sinclair in \cite{john1} have shown that every derivation on a semisimple Banach algebra is continuous. Ringrose in \cite{rin} showed that every derivation from a $C^{*}$-algebra $\A$ into Banach $\A$-bimodule $\U$ is continuous. In \cite{chr}, Christian proved that every derivation from a nest algebra on Hilbert space $\mathcal{H}$ into $\mathbb{B}(\mathcal{H})$ is continuous. Additionally, some results on automatic continuity of derivations on prime Banach algebras have been established by Villena in \cite{vi1} and \cite{vi2} .
\par 
The study of the first cohomology group of Banach algebras is also a considerable topic which may be used to study the structure of Banach algebras. Johnson in \cite{john2}, using the first cohomology group, has defined amenable Banach Algebras and then various types of amenability defined by using the first cohomology group. We may refer the reader to \cite{run} for more information. Also among the interesting problems in the theory of derivations is either characterizing algebras on which every continuous derivation is inner, that is, the first  cohomology group is trivial or characterizing the first cohomology group of Banach algebras up to vector spaces isomorphism. Sakai in \cite{sa} showed that every continuous derivation on a $W^{*}$-algebra is inner. Kadison \cite{ka} proved that every derivation of a $C^{*}$-algebra on a Hilbert space $\mathcal{H}$ is spatial (i.e., it is of the form $a\mapsto ta-at$ for $t\in \mathbb{B}(\mathcal{H})$) and in particular, every derivation on a von Neumann algebra is inner. Some results have also been obtained in the case of  non-self-adjoint operator algebras. Christian \cite{chr} showed that every continuous derivation on a nest algebra on $\mathcal{H}$ to itself and to $\mathbb{B}(\mathcal{H})$ is inner and then this result generalized to some other forms among which we may refer to \cite{li} and the references therein. Gilfeather and Smith have calculated the first cohomology group of some operator algebras called joins(\cite{gil1}, \cite{gil2}). In \cite{do}, the cohomology group of operator algebras called seminest algebras has been calculated. All of these operator algebras and nest algebras have a structure like triangluar Banach algebras, so motivated by these studies, Forrest and Marcoux in \cite{for} verified the first cohomology group of triangular Banach algebras. The first cohomology group of triangular Banach algebras is not necessarily zero and in \cite{for}, in fact it has been calculated under some special conditions and using the results, various examples of Banach algebras with non-trivial cohomology have been given and then the first cohomology group of those examples has been computed. Also in \cite{for}, some results concerning the automatic continuity of derivations on triangluar Banach algebras are presented. A generalization of triangluar Banach algebras is module extensions of Banach algebras which Zhang \cite{zh} has studied the weak amenability of them and then Medghalchi and Pourmahmood in \cite{med} computed the first cohomology group of module extensions of Banach algebras and using those results, they gave various examples of Banach algebras  with non-trivial cohomology group and in fact they calculated the first cohomology group of the given examples.
\par 
Another class of Banach algebras which considered during the last thirty years, is the class of algebras obtained by a special product called $\theta$-Lau product. This product is firstly introduced by Lau \cite{lau} for a special class of Banach algebras which are pre-dual of von Neumann algebras where the dual unit element (the unit element of the dual) is a multiplicative linear functional. Afterwards, various studies have been performed to it. For instance, Monfared in \cite{mn} has verified the structure of this special product and in \cite{gha} the amenability of Banach algebras equipped with the same product has been studied. For more information about this product the reader may refer to \cite{gha}, \cite{mn} and references therein.  
\par 
Let $\mathcal{B}$ be a Banach algebra such that $\mathcal{B} =\A\oplus\U$ (as direct sum of Banach spaces) where $\A$ is a closed subalgebra of $\mathcal{B}$ and $\U$ is a closed ideal in $\mathcal{B}$. In this case, we say that $\mathcal{B}$ is \textit{semidirect product} of $\A$ and $\U$ and write $\mathcal{B}=\A\ltimes \U$. Semidirect product Banach algebras appear in the study of many classes of Banach algebras. For instance, in the strong Wedderburn decomposition of a Banach algebra $\mathcal{B}$, it is assumed that $\mathcal{B}=\A\ltimes Rad\mathcal{B}$ where $Rad\mathcal{B}$ is the Jacobson radical of $\mathcal{B}$ and $\A $ is a closed subalgebra of $\mathcal{B}$ with $\A\cong \frac{\mathcal{B}}{Rad\mathcal{B}}$ or for example, in \cite{da2} using the structure of semidirect product, the authors have studied the amenability of measure algebras since every measure algebra has a decomposition to semidirect product of Banach algebras. In \cite{tho2}, Thomas has verified the necessary conditions of decomposition of a commutative Banach algebra to semidirect product of a closed subalgebra and a principal ideal. We also may refer to \cite{ba}, \cite{ber} and \cite{wh} where semidirect product Banach algebras are studied from different points of view. Equivalently, one may discuss semidirect product Banach algebras as follows; Let $\A$ and $\U$ be Banach algebras such that $\U$ is a Banach $\A$-bimodule with compatible actions and appropriate norm. Consider the multiplication on $\A\times\U$ given by 
\[(a,x)(b,y)=(ab,ay+xb+xy)\quad\quad  ((a,b)\in \A\times\U).\]
It can be shown that with this multiplication and $\l^{1}$-norm, $\A\times\U$ is a Banach algebra where $\A$ is a closed subalgebra of this Banach algebra and $\U$ is a closed ideal of it. So in fact this Banach algebra is equivalent to $\A\ltimes\U$. By considering different module actions or algebra multiplications, it can be seen that $\A\ltimes\U$ is a generalization of direct products of Banach algebras, tivial extension Banach algebras, triangular Banach algebras or $\theta$-Lau product Banach algebras. In this paper we consider semidirect product Banach algebras as mentioned above and study the derivations on this special product of Banach algebras. In fact we establish various results concerning the automatic continuity of derivations on $\A\ltimes\U$ and the first cohomology group of it and we present these results in special cases of $\A\ltimes\U$ and obtain various examples of Banach algebras with automatically continuous derivations and trivial first cohomology group or compute their first cohomology group.
\par
This paper is organized as follows. In section 2, we investigate the definition of semidirect product Banach algebras and we show that this product is a generalization of various types of Banach algebras. In section 3, the structure of derivations on semidirect product Banach algebras will be discussed and using that we obtain several results about the automatic continuity of derivations on these Banach algebras and also verify the decomposition of derivations into the sum of a continuous derivation and another derivation. In section 4, we consider the fist cohomology group of semidirect product Banach algebras and compute it under some different conditions and establish various results in this context. In section 5, we apply the obtained results in sections 3,4 to some special cases of semidirect products of Banach algebras. In fact we investigate the automatic continuity of the derivations and the first cohomology group of direct products of Banach algebras, module extension Banach algebras and $\theta$-Lau products of Banach algebras and  establish some various results about the derivations on these Banach algebras. 
\par 
At the end of this section we introduce some used notations and expressions in the paper.
\par 
If $\mathcal{X}$ and $\mathcal{Y}$ are Banach spaces, for a linear transform $T:\mathcal{X}\rightarrow \mathcal{Y}$, define the separating space $\mathfrak{S}(T)$ as 
\[\mathfrak{S}(T):=\{y\in \mathcal{Y}\, \mid \, \text{there is}\, \, \{x_n\}\subseteq \mathcal{X}\,\, \text{with}\, \, x_n\rightarrow 0 , \, T(x_n)\rightarrow y\} .\]
By closed graph theorem, $T$ is continuous if and only if $\mathfrak{S}(T)=(0)$.
\par 
Let $\A$ be a Banach algebra and $\U$ be a Banach $\A$-bimodule. By $Z(\A)$, we mean the center of $\A$. Consider the set $ann_{\A}\U$ as
\[ann_{\A}\U:=\{a\in\A\, \mid \, a\U=\U a=(0)\}.\]
If $\mathcal{N}$ is an $\A$-submodule of $\U$, we put 
\[(\mathcal{N}:\U)_{\A}:=\{a\in\A \, \mid \, a\U\subseteq \mathcal{N}, \, \U a\subseteq \mathcal{N}\}. \] 
It is clear that if $\mathcal{N}=(0)$, then $((0):\U)_{\A}=ann_{\A}\U$. Let $\U$ and $\mathcal{V}$ be Banach $\A$-bimodules. A linear map $\phi :\U\rightarrow \mathcal{V}$ is said to be a \textit{left $\A$-module homomorphism} if $\phi (ax)=a\phi (x)$ whenever $a\in \A$ and $x\in\U$ and it is a \textit{right $\A$-module homomorphism} if $\phi (xa)=\phi (x) a\quad (a\in \A,x\in\U)$. The linear map $\phi$ is called \textit{$\A$-module homomorphism}, if $ \phi$ is both of left and right $\A$-module homomorphism. The set of all continuous $\A$-module homomorphisms from $\U$ into $\mathcal{V}$ is denoted by $Hom_{\A}(\U,\mathcal{V})$. Note that if the spaces are the same, we just write $Z^{1}(\A), N^{1}(\A), H^{1}(\A), Hom_{\A}(\U)$
\section{Semidirect products of Banach algebras}
In this section we introduce the notion of semidirect produts of Banach algebras and give some properties of this concept.\\
Let $\A$ and $\U$ be Banach algebras such that $\U$ is a Banach $\A$-bimodule with
\[\parallel ax \parallel \leq \parallel a \parallel \parallel x \parallel, \quad \parallel xa \parallel \leq \parallel x \parallel \parallel a \parallel \quad\quad (a\in \A, x\in \U), \] 
and compatible actions, that is 
\[ (a.x)y=a.(xy),\,\,\, (xy).a=x(y.a), \,\,\, (x.a)y=x(a.y)  \quad\quad (a \in \A, \, x,y\in \U). \]
If we equip the set $\A \times \U$ with the usual $\mathbb{C}$-module structure, then the multiplication 
\[ (a,x)(b,y)=(ab, a.y+x.b+xy) \]
turns $\A \times \U$ into an associative algebra. 
\par 
In continue we also denote the module actions by $ax$ and $xa$.
\par 
The \textit{semidirect product} of Banach algebras $\A$ and $\U$, denoted by $\A \ltimes \U$, is defined as the space $\A \times \U$ with the above algebra multiplication and with the norm 
\[ \parallel (a,x) \parallel = \parallel a \parallel + \parallel x \parallel . \]
The semidirect product $\A \ltimes \U$ is a Banach algebra.
\begin{rem}\label {1}
In $\A \ltimes \U$ we identify $\A \times \lbrace 0 \rbrace$ with $\A$, and $\lbrace 0 \rbrace \times \U $ with $\U$. Then $\A$ is a closed subalgebra while $\U$ is a closed ideal of $\A \ltimes \U$, and
\[ \A \ltimes \U / \U \cong \A \quad (isometric \, \, isomorphism). \]
Indeed $\A \ltimes \U$ is equal to direct sum of $\A$ and $\U$ as Banach spaces.
\par
Conversely, let $\mathcal{B}$ be a Banach algebra which has the form $\mathcal{B}=\A \oplus \U$ as Banach spaces direct sum, where $\A$ is a closed subalgebra of $\mathcal{B}$ and $\U$ is a closed ideal in $\mathcal{B}$. In this case, the product of two elements $a+x$ and $b+y$ of  $\mathcal{B}=\A \oplus \U$ is given by 
\[(a+x)(b+y)=ab+(ay+xb+xy), \] where $ab\in \A$ and $ay+xb+xy\in \U$. Also the norm on $\mathcal{B}$ is equivalent to the one given by 
\[ \parallel a+x \parallel =\parallel a \parallel + \parallel  x \parallel \quad\quad (a\in \A, x \in \U)).\]
On the other hand $\A$ and $\U$ are Banach algebras such that  $\U$ is a Banach $\A$-bimodule with compatible actions and
\[\parallel ax \parallel \leq \parallel a \parallel \parallel x \parallel, \quad \parallel xa \parallel \leq \parallel x \parallel \parallel a \parallel \quad\quad (a\in \A, x\in \U). \] 
By above arguments we have 
\[ \mathcal{B}\cong \A \ltimes \U , \]
as isomorphism of Banach algebras.
\end{rem}
The following examples provides various
types of semidirect product af Banach algebras.
 \begin{exm}
Let $\A$ be a Banach algebra. Then the unitization of $\A$ denoted by $\A^{\#}$ is in fact the semidirect product  $\mathbb{C}\ltimes \A$.
\end{exm}
\begin{exm}\label{dp}
Suppose that the action  $\A$ on  $\U$ is trivial, that is, $ \A\U=\U\A=(0)$, then we obtain the usual $l^{1}$-direct product of Banach algebras $\A$ and $\U$. In this case, $\A \ltimes \U= \A \times \U$.
\end{exm}
\begin{exm}\label{me}
Let the algebra multiplication on $\U$ be the trivial action, that is, $\U^2=(0)$, then $\A\ltimes \U$ is same as the module extension (or trivial extension) of $\A$ by $\U$ which we denote by $T(\A,\U)$.
\end{exm}
\begin{exm}\label{tri}
Let $\A$ and $\mathcal{B}$ be Banach algebras and $\mathcal{M}$ be a Banach $(\A,\mathcal{B})$-bimodule. The triangular Banach algebra introduced in \cite{for} is 
$Tri(\A, \mathcal{M}, \mathcal{B}):=\begin{pmatrix}
 \A & \mathcal{M}  \\
   0 &\mathcal{B} 
\end{pmatrix}$ with the usual matrix operations and $l^1$-norm. If we trun $\mathcal{M}$ into a  Banach $\A \times\mathcal{B}$-bimodule with the actions 
$$(a,b)m=am\quad ,\quad m(a,b)=mb\quad\quad ((a,b)\in \A\times\mathcal{B}, m\in \mathcal{M} )$$
( $\A \times \mathcal{B}$  is considered with $l^1$-norm), then $Tri(\A, \mathcal{M}, \mathcal{B})\cong T(\A\times\mathcal{B},\mathcal{M})$ as isomorphism of Banach algebras. Therefore triangular Banach algebra is an example of semidirect products of Banach algebras.
\end{exm}
\begin{rem}
Let $\A$ and $\U$ be Banach algebras such that $\U$ is a Banach $\A$-bimodule with the compatible actions and norm, and  let $\overline{\A\U}=\U$ or $\overline{\U\A}=\U$. Let the Banach algebra $\A$ be the direct sum of its closed ideals $I_1$ and $I_2$, that is, $\A=I_1 \oplus I_2$. If $I_2\U=\U I_1=(0)$, then $\U^2=(0)$. Because if $a\in I_1$, $b\in I_2$ and $x,y\in \U$ are arbitrary, then by the associativity we have $$x[(a+b)y]=[x(a+b)]y.$$
So $$x(ay)=(xb)y.$$
If $\overline{\A\U}=\U$ and we put $b=0$ in the above equation, then $xy=0$ and analogously, if $\overline{\U\A}=\U$, letting $a=0$ gives  $xy=0$.
\par 
By the above arguments, hypothesis and according to Example \ref{tri}, in this case we have 
\begin{equation*}
\A\ltimes \U=T(\A,\U)\cong\begin{pmatrix}
 I_1 & \U \\
   0 &I_2
\end{pmatrix}
\end{equation*}
 as isomorphism of Banach algebras.
\end{rem}
\begin{exm}\label{la}
Let $\A$ and $\U$ be Banach algebras, $\theta\in \Delta(\A)$ where $\Delta(\A)$ is the set of all non-zero characters of $\A$. With the following module actions, $\U$ becomes a Banach  $\A$-bimodule:
$$ax=xa=\theta(a)x\quad\quad (a\in \A ,x\in \U).$$
The norm and the actions on $\U$ are compatible and one can consider $\A\ltimes \U$ with the multiplication
$$(a,x)(b,y)=(ab,\theta(a)y+\theta(b)x+xy).$$
In this case $\A\ltimes \U$ is the $\theta$-Lau product which is introduced in \cite{lau}. 
\end{exm}
\begin{exm}
Let $G$ be a locally compact group. Then $M(G)=l^{1}(G)\ltimes M_{c}(G)$ where $M_{c}(G)$ is a subspace of $M(G)$ containing all continuous measures such that for $\mu\in M(G)$, $\mu\in M_c(G)$ if and only if $\mu(\{s\})=0$ $(s\in G)$. We denote the subspace of discrete measures by $M_d({G})$ which is isomorphic to $l^{1}(G)$ and 
\begin{equation*}M_{d}(G)=\{\mu =\sum{\alpha_{s}\delta_{s}}:\Vert\mu\Vert =\sum_{s\in G}{\vert \alpha_{s}\vert }<\infty \}.
\end{equation*}
Indeed $l^{1}(G)$ is a closed subspace of $M(G)$ and $M_c(G)$ is a closed ideal of $M(G)$. If $G$ is discrete, then $M(G)=l^{1}(G)$ and $M_{c}(G)=\{0\}$. But if $G$ is not discrete, then $M_{c}(G)\neq\{0\}$ 
\end{exm}
It is possible for an algebra $\mathcal{B}$ with a closed ideal $\mathcal{I}$, not to exist some closed subalgebra $\A$ of $\mathcal{B}$ such that $\mathcal{B}=\A\ltimes \mathcal{I}$. In the following we give an example of such a Banach algebra.
\begin{exm}
Let $\mathcal{B}:=C([0,1])$ be the Banach algebra of continuous complex-valued functions on $[0,1]$ and let $\mathcal{I}:=\{f\in \mathcal{B} :f(0)=f(1)=0\}$. $\mathcal{I}$ is a closed ideal of $\mathcal{B}$. If $\A $ is a closed subalgebra of $\mathcal{B}$ satisfying  $\mathcal{B}=\A \oplus \mathcal{I}$ as Banach spaces direct sum, then for $f\in \mathcal{I}$ and $g\in \A$ with $f(x)+g(x)=x$ for $x\in [0,1]$, we have $g(0)=0$, $g(1)=1$ and $g-g^2\in \A\cap \mathcal{I}$. But yet $g-g^2\neq 0$.
\end{exm}
Note that $\A\ltimes\U$ is commutative if and only if both $\A$ and $\U$ are commutative Banach algebras and $\U$ is a commutative $\A$-bimodule.
\par 
In the rest of this section we introduce some special maps which are used in next sections.
\par 
For $a\in\A$ define the map $r_{a}:\U\rightarrow\U$ by $r_{a}(x)=xa-ax$. Some properties of this map are given in the following remark.
\begin{rem}\label{inn1}
For $a\in\A$, consider the map $r_{a}:\U\rightarrow\U$.
\begin{enumerate}
\item[(i)]
$r_a$ is a derivation on $\U$.
\item[(ii)]
For every $b\in\A$ and $x\in\U$,
\[r_{a}(bx)=br_{a}(x)+id_{a}(b)x \quad \text{and} \quad r_{a}(xb)=r_{a}(x)b+x id_{a}(b).\]
\item[(iii)]
For $a\in \A$, if $id_{a}=0$, then $r_a$ is an $\A$-bimodule homomorphism. Also if $ann_{\A}\U =(0)$ and $r_a$ is an $\A$-bimodule homomorphism, then $id_{a}=0$.
\end{enumerate}
\end{rem}
Also inner derivations on $\U$ have significant properties considered in determining the first cohomology group of $\A\ltimes\U$ which are given in the next remark.
\par 
Note that both of inner derivations from $\A$ to $\U$ and inner derivations from $\U$ to $\U$  are denoted by $id_{x}$. So in order to avoid confusion, we denote by $id_{\A , x}$, the inner derivations  from $\A$ to $\U$ while $id_{\U , x}$ denotes the inner derivations from $\U$ to $\U$.
\begin{rem}\label{inn2}
For $x_0\in\U$, consider the inner derivation $id_{\U , x_0}:\U\rightarrow\U$.
\begin{enumerate}
\item[(i)]
For every $a\in\A$ and $x\in\U$, 
\[id_{\U ,x_0}(ax)=a \, id _{\U ,x_0}(x)+id _{\A , x_0}(a)x \quad \text{and} \quad id_{\U ,x_0}(xa)=id _{\U ,x_0}(x)a+x \, id _{\A, x_0}(a).\]
\item[(ii)]
For $x_{0}\in \U$, if $id_{\A, x_0}=0$, then $id_{\U , x_0}$ is an $\A$-bimodule homomorphism. If $ann_{\U}\U =(0)$ and $id_{\U , x_0}$ is an $\A$-bimodule homomorphism, then $id_{\A, x_0}=0$.
\end{enumerate}
\end{rem}
This following sets play an important role in determining the first cohomology group of $\A\ltimes\U$.
\[R_{\A}(\U):=\{r_{a}: \U\rightarrow\U \,  \mid  \, a\in \A\};\]
\[C_{\A}(\U):=\{r_{a} :\U\rightarrow\U \,  \mid  \,  id_{a}=0 \, \, (a\in \A)\};\]
\[I(\U):=\{id_{\U, x}:\U\rightarrow\U \,  \mid  \,  id_{\A, x}=0 \, \, (x\in \U)\};\]
\indent In view of the above remarks, the set $R_{\A}(\U)$ is a linear subspace of $Z^{1}(\U)$. Also $C_{\A}(\U)$ is a linear subspace of $Hom_{\A}(\U) \cap R_{\A}(\U)$ and $I(\U)$ is a linear subspace of $Hom_{\A}(\U) \cap N^{1}(\U)$. Indeed, we have the following inclusion linear subspaces;
\[ C_{\A}(\U)+I(\U)\subseteq Hom_{\A}(\U) \cap (R_{\A}(\U)+N^{1}(\U))\subseteq Hom_{\A}(\U) \cap Z^{1}(\U).\]
If $\A$ is commutative, then $R_{\A}(\U)= C_{\A}(\U)$ and if $\U$ is a commutative $\A$-bimodule, then $N^{1}(\U)=I(\U)$. If $\U ^{2}=(0)$, then $Z^{1} (\U)=\mathbb{B}(\U)$ and $N^{1}(\U)=I(\U )=(0)$.
\section{Derivations on $\A \ltimes \U$}
In this section we determine the structure of derivations on $\A\ltimes \U$. According to which, we get some results concerning the automatic continuity of derivations on $\A\ltimes \U$. Also we use the results of this section to determine the first cohomology group of $\A\ltimes \U$ in the next sections.
\par 
Throughout this section we always assume that $\A$ and $\U$ are Banach algebras where $\U$ is a Banach $\A$-bimodule with the compatible actions and norm (As in the section 2). In other cases the conditions will be specified.
\par 
In the following theorem the structure of derivations on $\A\ltimes \U$ are determined.
\begin{thm}\label{asll}
Let $D:\A\ltimes \U\rightarrow \A\ltimes \U$ be a map. Then the following conditions are equivalent.
\begin{enumerate}
\item[(i)] $D$ is a derivation.
\item[(ii)] 
\[D((a,x))=(\delta_1 (a)+\tau_1 (x),\delta_2 (a)+\tau_2 (x))\quad\quad (a\in \A,x\in \U)\]
such that 
\begin{enumerate}
\item[(a)] 
$\delta_1:\A\rightarrow\A$ is a derivation.
\item[(b)]
$\delta_2:\A\rightarrow \U$ is a derivation.
\item[(c)]
$\tau_1 :\U\rightarrow \A$ is an $\A$-bimodule homomorphism such that 
$$\tau_1(xy)=0,$$
for all $x,y\in \U$.
\item[(d)]
$\tau_2 :\U\rightarrow \U$ is a linear map such that for every $a\in \A$ and $x,y\in \U$ satisfies the following conditions
\begin{eqnarray*}
\tau_2 (ax)&=&a\tau_2 (x)+\delta_1 (a)x+\delta_2 (a)x; \\
\tau_2 (xa)&=&\tau_2 (x)a+x\delta_1 (a)+x\delta_2 (a); \\
\tau_2(xy)&=&x\tau_1(y)+\tau_1(x)y+x\tau_2(y)+\tau_2(x)y.
\end{eqnarray*}
\end{enumerate}
\end{enumerate}
Moreover, $D$ is an inner derivation if and only if $\delta_1 ,\delta_2$ are inner derivations, $\tau_1 =0$ and if $\delta_1 =ad_{a_0}$ and $ \delta_2 =ad_{x_0}$, then $\tau_2 =ad_{x_0}+r_{a_0}$.
\end{thm}
\begin{proof}
$(i)\implies (ii)$. Since $D$ is a linear map and $\A\ltimes \U$ is the direct sum of the linear spaces $\A$ and $\U$, there are some linear maps $\delta_1:\A\rightarrow \A$, $\delta_2:\A\rightarrow \U$, $\tau_1 :\U\rightarrow \A$ and  $\tau_2 :\U\rightarrow \U$ such that for any $(a,x)\in \A\ltimes \U$ we have 
$$D((a,x))=(\delta_1 (a)+\tau_1 (x),\delta_2 (a)+\tau_2 (x)).$$
By applying $D$ on the equality $(a,0)(b,0)=(ab,0)$ we conclude that $\delta_1$ and $\delta_2$ are derivations. Analogously, applying $D$ on the equalities 
\[ (a,0)(0,x)=(0,ax),\,\,\, (0,x)(a,0)=(0,xa) \,\,\, and \,\,\, (0,x)(0,y)=(0,xy),\]
 establishes the desired properties for the maps $\tau_1$ and $\tau_2$ given in parts $(c)$ and $(d)$ respectively.
$(ii)\implies (i)$ Clear.
\par 
The equivalent conditions of inner-ness of $D$ can be obtained by a straightforward calculation.
\end{proof}
In the sequel for a derivation $D$ on $\A\ltimes \U$, we always assume that 
$$D((a,x))=(\delta_1 (a)+\tau_1 (x),\delta_2 (a)+\tau_2 (x))\quad\quad ((a,x)\in \A\ltimes \U)$$ in which the mentioned maps satisfy the conditions of the preceding theorem.
\par 
By Theorem \ref{asll}, if $D$ is an inner derivation on $\A\ltimes \U$, then $\tau_1 =0$. So this question is of interest for a given derivation $D$, under what conditions one has $\tau_1 =0$?
By part $(ii)-(c)$ of Theorem \ref{asll} if $\U^2=\U$ ($\overline{\U^2}=\U$, if $D$ is continuous), then $\tau_1 =0$. If $\U$ has a bounded approximate identity, then by Cohen's factorization theorem we have $\U^2=\U$ and therefore in this case $\tau_1 =0$. All unital Bnach algebras, $C^{*}$-algebras and group algebras have bounded approximate identity. Also if $\U$ is a simple Banach algebra, then $\U^{2}=\U$.
\par 
 The following corollary follows from Theorem \ref{asll}.
 \begin{cor}\label{tak}
 Suppose that $\delta_1:\A\rightarrow \A$,  $\delta_2:\A\rightarrow \U$, $\tau_1 :\U\rightarrow \A$ and $\tau_2 :\U\rightarrow \U$ are linear  maps.
 \begin{enumerate}
 \item[(i)]
 $D:\A\ltimes\U\rightarrow \A\ltimes\U $ defined by $D((a,x))=(\delta_1(a),0)$ is a derivation if and only if $\delta_1$ is a derivation and $\delta_1 (\A)\subseteq ann_{\A}\U$. In this case if $\delta_1 =id_{a_0}$ where $aa_0-a_0 a\in ann_{\A}\U \, \, (a\in \A) $ and for any $x\in \U$, $a_0x =xa_0$, then $D$ is inner.
 \item[(ii)]
  $D:\A\ltimes\U\rightarrow \A\ltimes\U $ with $D((a,x))=(0,\delta_2(a))$ is a derivation if and only if $\delta_2$ is a derivation and $\delta_2 (A)\subseteq ann_{\U}\U$. Moreover, if $\delta_2 =id_{\A, x_0}$ is inner, $ax_0 - x_0a\in ann_{\U}\U$ (for all $a\in \A$) and $x_0\in Z(\U)$, then $D$ is inner.
 \item[(iii)]
$D:\A\ltimes\U\rightarrow \A\ltimes\U $ with $D((a,x))=(\tau_1(x),0)$ is a derivation if and only if $\tau_1(xy)=0, x\tau_1(y)+\tau_1(x)y=0 (x,y\in\U)$. In this case $D$ is inner if and only if $\tau_1 =0$.
\item[(iv)] $D:\A\ltimes\U\rightarrow \A\ltimes\U $ with $D((a,x))=(0,\tau_2(x))$ is a derivation if and only if $\tau_2$ is a derivation and also an $\A$-bimodule homomorphism. In this case $D$ is inner if and only if $\tau_2(x)=r_{a_0}(x)+id_{\U, x_0}(x)$ where $a_0\in Z(\A)$ and $ax_0 =x_0 a$ for all $a\in \A$.
\end{enumerate}
 \end{cor}
In the following we give an example showing that the condition $\tau_1 =0$ does not hold in general.
\begin{exm}
Let $\mathcal{B}$ be a Banach algebra and $\A:=T(\mathcal{B},\mathcal{B})$ and let $\U :=\mathcal{B}$. The Banach algebra $\U$ becomes a Banach $\A$-bimodule by the following compatible module actions:
\[(a,b)x=ax,\quad  x(a,b)= xa\quad\quad (a,b\in \A , x\in \U).\]
Now consider the Banach algebra $T(\A,\U)$ and define the map $\tau_1:\U\rightarrow \A$ by 
$$\tau_1(x)=(0,x)\quad\quad (x\in \U)$$
Then $\tau_1\neq 0$ is an $\A$-bimodule homomorphism with $\tau_1(\U ^2)=(0)$ and for any $x,y\in \U$ we have 
$$x\tau_1(y)+\tau_1(x)y=x(0,y)+(0,x)y=0$$
Now by Corollary \ref{tak} it can be seen that the map $D$ on $T(\A,\U)$ defined by $D((a,b),x)=(\tau_1(x),0)$ is a derivation in which $\tau_1\neq 0$.
\end{exm}
In this example $\tau_1\neq 0$ but $x\tau_1(y)+\tau_1(x)y=0$ $ (x,y\in\U)$. The next example shows that $\tau_1$ does not satisfy the condition $x\tau_1(y)+\tau_1(x)y=0 \,\, (x,y\in\U)$ in general.
\begin{exm}
Let $\A$ be a Banach algebra, $\mathcal{C}$ be a Banach $\A$-bimodule and $\gamma :\mathcal{C}\rightarrow \A$ be a nonzero $\A$-bimodule homomorphism such that 
\[ c\gamma (c')+\gamma (c)c' =0\]
for all $c,c'\in\mathcal{C}$.
\\
Let $\U :=\A\times \mathcal{C}$. With the usual actions, $\U$ is a Banach $\A$-bimodule. Consider the multiplication on $\U$ given by
\[(x,y)(x',y')=(xx',0)\quad\quad ((x,y),(x',y')\in\U). \]
With this multiplication $\U$ is a Banach algebra such that the multiplication on $\U$ is compatible with its module actions. Now we may consider the Banach algebra $\A\ltimes\U$. Define the maps $\tau_1:\U\rightarrow\A$ and $\tau_2:\U\rightarrow\U$ by 
\[\tau_1 ((x,y))=\gamma(y)\quad , \quad \tau_2((x,y))=(-\gamma (y),0).\]
$\tau_1$ and $\tau_2$ are $\A$-bimodule homomorphisms such that $\tau_1(\U ^2)=0$, $\tau_1\,,\tau_2\neq 0$ and 
\[0=\tau_2((x,y)(x',y'))=(x,y)\tau_1 ((x',y'))+(x,y)\tau_2((x',y'))+\tau_1 ((x,y)) (x',y')+\tau_2((x,y))(x',y').\]
By Theorem \ref{asll} it can be seen that $D=(\tau_1,\tau_2)$ is a derivation on $\A\ltimes\U$. If we assume further that $\A$ is a unital algebra, then for any $y,y'\in\mathcal{C}$ with $\gamma (y)\neq 0$,
\begin{eqnarray*}
(0,y)\tau_1((1,y'))+\tau_1((0,y))(1,y')&=&(0,y)\gamma (y')+\gamma (y)(1,y')\\&=&(\gamma(y),y\gamma (y')+\gamma(y)y')\\&=&(\gamma (y),0)\neq 0.
\end{eqnarray*}
 \end{exm}
In the following we investigate the automatic continuity of derivations on $\A\ltimes\U$. It is clear that a derivation $D$ on $\A\ltimes \U$ is continuous if and only if the maps $\delta_1 , \delta_2 , \tau_1$ and $\tau_2$ are continuous. 
\par 
 In the next theorem we state the relation between separating spaces of $\delta_i$ and $\tau_i$ ($i=1,2)$ and then by using this theorem, we obtain some results concerning the automatic continuity of the derivations on $\A\ltimes\U$.
\begin{thm}\label{joda}
Let $D$ be a derivation on $\A\ltimes\U$ such that 
$$D((a,x))=(\delta_1 (a)+\tau_1 (x),\delta_2 (a)+\tau_2 (x))\quad\quad ((a,x)\in \A\ltimes \U),$$
then
\begin{enumerate}
\item[(i)]
$\C (\tau_1)$ is an ideal in $\A$. If $\tau_2$ is continuous, then $\C (\tau_1)\subseteq ann_{\A}\U$ and if $\tau_2(\U)\subseteq ann_{\U}\U$, then $\C (\tau_1)\subseteq(\C (\tau_2):\U)_\A$.
\item[(ii)]
$\C (\tau_2)$ is an $\A$-sub-bimodule of $\U$ and if $\tau_1$ is continuous or $\tau_1 (\U)\subseteq ann_{\A}\U$, then $\C (\tau_2)$ is an ideal in $\U$.
\item[(iii)]
$\C (\delta_1)$ is an ideal in $\A$ and if $\delta_2$ is continuous or $\delta_2 (\A)\subseteq ann_{\U}\U$, then $\C (\delta_1)\subseteq(\C (\tau_2):\U)_\A$.
\item[(iv)]
$\C (\delta_2)$ is an $\A$-subbimodule of $\U$ and if $\delta_1$ is continuous or $\delta_1 (\A)\subseteq ann_{\A}\U$, then $\C (\delta_2)\subseteq(\C (\tau_2):\U)_\U$.
\end{enumerate}
\end{thm}
\begin{proof}
$(i)$ Let $a\in \C (\tau_1)$. Thus there is a sequence $\{x_n\}$ in $\U$ such that 
$x_n\rightarrow 0$ and $ \tau_1(x_n)\rightarrow a$.
For any $b\in \A$, $bx_n\rightarrow 0$, so $\tau_1(bx_n)=b\tau_1 (x_n)\rightarrow ba $ and hence $ba\in \C (\tau_1)$. Similarly, $ab\in \C (\tau_1)$. Therefore $\C (\tau_1)$ is an ideal in $\A$.\\
 Now for every $x\in \U$ we have  
$$\tau_2(xx_n)=x\tau_1(x_n)+x\tau_2(x_n)+\tau_1(x)x_n+\tau_2(x)x_n$$
and
$$\tau_2(x_nx)=x_n\tau_1(x)+x_n\tau_2(x)+\tau_1(x_n)x+\tau_2(x_n)x.$$
If $\tau_2$ is continuous, taking limit of the above equations gives $ax=xa=0$ ($x\in \U$). So $a\in ann_{\A}\U$. Hence $\C (\tau_1)\subseteq ann_{\A}\U$. 
\par 
If $\tau_2 (\U)\subseteq ann_{\U}\U$, by taking limit of the above equations we get 
$$\tau_2(xx_n)\rightarrow xa\quad  \text{and} \quad \tau_2(x_n x)\rightarrow ax.$$ So $ax,xa\in \C (\tau_2)$ and hence $\C (\tau_1)\subseteq(\C (\tau_2):\U)_\A$.
\\
$(ii)$ If $x_n\rightarrow 0$ and $\tau_2(x_n)\rightarrow x$, then for any $a\in \A$ we have 
$$\tau_2(ax_n)=a\tau_2(x_n)+\delta_1(a)x_n+\delta_2(a)x_n$$
and
$$\tau_2(x_{n}a)=\tau_2(x_n)a+x_n\delta_1(a)+x_n\delta_2(a).$$
By taking limits of these equations we get 
$$\tau_2(ax_n)\rightarrow ax\quad\text{and}\quad \tau_2(x_{n}a)\rightarrow xa.$$
So $ax,xa\in \C (\tau_2)$ and hence $\C (\tau_2)$ is an $\A$-subbimodule of $\U$.
\par 
For any $y\in \U$ we have
$$\tau_2(yx_n)=\tau_1(y)x_n+\tau_2(y)x_n+y\tau_1(x_n)+y\tau_2(x_n)$$
$$\tau_2(x_n y)=\tau_1(x_n)y+\tau_2(x_n)y+x_n\tau_1(y)+x_n\tau_2(y).$$
If $\tau_1$ is continuous or $\tau_1(\U)\subseteq ann_{\A} \U$, by taking limit we obtain 
$$\tau_2(yx_n)\rightarrow yx\quad \text{and}\quad \tau_2(x_{n}y)\rightarrow xy.$$
Therefore $xy,yx\in \C (\tau_2)$ and $\C (\tau_2)$ is an ideal.
\\
$(iii)$ Let $a_n\rightarrow 0 $ and $\delta_1 (a_n)\rightarrow a$. Since $\delta_1$ is a derivation, it follows that $\C (\delta_1)$ is an ideal in $\A$. For every $x\in \U$ we have 
$$\tau_2(a_nx)=a_n\tau_2(x)+\delta_1(a_n)x+\delta_2(a_n)x$$
and
$$\tau_2(ax_n)=\tau_2(x)a_n+x\delta_1(a_n)+x\delta_2(a_n).$$
If $\delta_2$ is continuous or $\delta_2(\A)\subseteq ann_{\U}\U$, taking limit of the above equations yields 
$$\tau_2(a_n x)\rightarrow ax\quad \text{and}\quad \tau_2(x a_n)\rightarrow xa.$$
Since $xa_n, a_nx\rightarrow 0$, it follows that $ax,xa\in \C (\tau_2)$. So $\C (\delta_1)\subseteq(\C (\tau_2):\U)_\A$.
\\
$(iv)$ The proof is similar to part $(iii)$.
\end{proof}
Now we give some results of the preceding theorem.
\begin{prop}\label{au1}
Let $D$ be a derivation on $\A\ltimes\U$ such that 
$$D((a,x))=(\delta_1 (a)+\tau_1 (x),\delta_2 (a)+\tau_2 (x))\quad\quad ((a,x)\in \A\ltimes \U),$$
then 
\begin{enumerate}
\item[(i)]
if $ann_{\A}\U =(0)$, $\tau_2$ and $\delta_2$ are continuous, then $D$ is continuous.
\item[(ii)]
if $ann_{\U}\U =(0)$, $\tau_2$ and $\delta_1$ are continuous, then $\delta_2$ is continuous. If we also add the assumption of continuity of $\tau_1$, then $D$ is continuous.
\end{enumerate}
\end{prop}
\begin{proof}
$(i)$ Since $\tau_2$ is continuous so $\C (\tau_2)=0$. By Theorem \ref{joda}-$(i)$ from the continuity of $\tau_2$ we conclude that $\C (\tau_1)\subseteq ann_{\A}\U=(0)$, thus $\tau_1$ is continuous. Moreover, since the conditions of Theorem \ref{joda}-$(iii)$ hold, we have
$$\C (\delta_1)\subseteq (\C (\tau_2):\U)_\A=((0):\U)_\A =ann_{\A}\U=(0).$$ Therefore $\delta_1$ is continuous. So $D$ is continuous.
\\
$(ii)$ Since the conditions of Theorem \ref{joda}-$(iv)$ hold and $\tau_2$ is continuous, it follows that
 $$\C (\delta_2)\subseteq (\C (\tau_2):\U)_\U=((0):\U)_\U=ann_{\U}\U =(0).$$ So $\delta_2$ is continuous . If $\tau_1$ is continuous, then we have the continuity of all the maps and thus $D$ is continuous.
\end{proof}
If we add the condition $ann_{\A}\U=(0)$ to part $(ii)$ of the previous proposition, then as in part $(i)$, this implies that $\tau_1$ is continuous and in this case $D$ is continuous as well.
\par 
Now by the preceding proposition, we obtain some results concerning the automatic continuity of the derivations on $\A\ltimes\U$.
\begin{cor}\label{n1}
Suppose that for every derivation $D$ on $\A\ltimes \U$ we have 
$x\tau_1 (y)+\tau_1(x)y=0$ $ (x,y\in \U)$. If $ann_{\A}\U=(0)$ and each derivation from $\U$ to $\U$ and each derivation from $\A$ to $\U$ is continuous, then every derivation on $\A\ltimes \U$ is continuous.
\end{cor}
\begin{proof}
By Theorem \ref{asll}, for any derivation $D=(\delta_1 +\tau_1,\delta_2+\tau_2)$ on $\A\ltimes \U$, the maps $\delta_2:\A \rightarrow\U$ and $\tau_2:\U\rightarrow \U$ are derivations, since we have 
$x\tau_1 (y)+\tau_1(x)y=0$ $ (x,y\in \U)$.  So $\delta_2$ and $\tau_2$ are continuous by the hypothesis. Now the continuity of $D$ follows from Proposition \ref{au1}-$(i)$.
\end{proof}
\begin{cor}\label{n2}
Suppose that for any derivation $D$ on $\A\ltimes \U$ we have  $x\tau_1 (y)+\tau_1(x)y=0 $ $(x,y\in \U)$. If $ann_{\U}\U=(0)$, every derivation from $\U$ to $\U$ and every derivation from $\A$ to $\A$ is continuous and every $\A$-bimodule homomorphism from $\U$ to $\A$ is continuous, then every derivation on $\A\ltimes \U$ is continuous.
\end{cor}
\begin{proof}
Consider the derivation $D=(\delta_1 +\tau_1,\delta_2+\tau_2)$. By Theorem \ref{asll}, the maps 
 $\delta_1:\A\rightarrow \A$ and  $\tau_2:\U\rightarrow \U$ are derivation and  $\tau_1:\U\rightarrow \A$ is an $\A$-bimodule homomorphism. By the hypothesis, $\delta_1, \tau_1$ and $\tau_2$ are continuous. Therefore  $D$ is continuous by Proposition \ref{au1}-$(ii)$.
\end{proof}
By Johnson's and Sinclair's theorem \cite{john1}, every derivation on a semisimple Banach algebra is continuous, so we have the following corollary.
\begin{cor}\label{semi}
Suppose that $\A$ and $\U$ are semisimple Banach algebras such that $\U$ has a bounded approximate identity. Then every derivation on $\A\ltimes \U$ is continuous.
\end{cor}
\begin{proof}
Let $D=(\delta_1 +\tau_1,\delta_2+\tau_2)$ be a derivation on $\A\ltimes \U$. By the Cohen's factorization theorem, $\U ^2 =\U$ and so $\tau_1=0$. Also from the hypothesis we conclude that $ann_{\U}\U =(0)$. By Johnson's and Sinclair's theorem \cite{john1}, every derivation on $\A$ and $\U$ is continuous. Now by Proposition \ref{au1}-$(ii)$, every derivation on $\A\ltimes \U$ is continuous.
\end{proof}
All $C^{*}$-algebras, semigroup algebras, measure algebras and unital simple algebras are semisimple Banach algebras with bounded approximate identity. 
\par 
In next results we investigate some conditions under which we can express the derivation $D$ as the sum of two derivations one of which being continuous.
\begin{prop}\label{taj1}
Let $D$ be a derivation on $\A\ltimes\U$ such that 
\[D((a,x))=(\delta_1 (a)+\tau_1 (x),\delta_2 (a)+\tau_2 (x))\quad\quad ((a,x)\in \A\ltimes \U),\]
then
\begin{enumerate}
\item[(i)]
if $ann_{\U}\U =(0)$, $x\tau_1 (y)+\tau_1(x)y=0 $ $(x,y\in \U)$, $\delta_1$ and $\tau_2$ are continuous, then $D=D_1 +D_2$ in which 
$$D_1((a,x))=(\delta_1(a),\delta_2(a)+\tau_2(x))\quad\text{and}\quad D_2((a,x))=(\tau_1(x),0),$$
such that $D_1$ and $D_2$ are derivations on $\A\ltimes\U$ and $D_1$ is continuous.
\item[(ii)]
if $ann_{\U}\U =(0)$, $\delta_1(\A)\subseteq ann_{\A}\U$ and $\tau_1, \tau_2$ are continuous, then \[D_1((a,x))=(\tau_1(x),\delta_2 (a)+\tau_2 (x))\] is a continuous derivation on $\A\ltimes \U$ and $D=D_1 +D_2$ where $D_2((a,x))=(\delta_1(a),0)$ is a derivation on $\A\ltimes \U$.
\end{enumerate}
\end{prop}
\begin{proof}
$(i)$ By Proposition \ref{au1}-$(ii)$, $\delta_2$ is continuous. By Corollary \ref{tak}, $D=D_1 +D_2$ where 
$$D_1((a,x))=(\delta_1(a),\delta_2(a)+\tau_2(x)) \quad \text{and} \quad D_2((a,x))=(\tau_1(x),0)\quad\quad ((a,x)\in \A\ltimes\U)$$
are derivations on $\A\ltimes \U$. By the assumptions and that $\delta_2$ is continuous, it follows that $D_1$ is a continuous derivation.
\\
$(ii)$ By Corollary \ref{tak} 
$$D_1((a,x))=(\tau_1(x),\delta_2 (a)+\tau_2 (x))\quad\text{and}\quad D_2((a,x))=(\delta_1(a),0)$$
are derivations on $\A\ltimes \U$. Now by the continuity of $\tau_2$, $ann_{\U}\U =(0)$ and that $\delta_1(\A)\subseteq ann_{\A}\U$,  from Theorem \ref{joda}-$(iv)$, it follows that 
$$\C (\delta_2)\subseteq (\C (\tau_2):\U)_\U =((0):\U)_\U =ann_{\U}\U=(0).$$ 
Therefore $\delta_2$ is continuous and hence so is $D_1$.
\end{proof}
To prove the next proposition we need the following lemma.
\begin{lem}(\cite [Proposition 5.2.2]{da}).\label{da}
Let $\mathcal{X}$, $\mathcal{Y}$ and $\mathcal{Z}$ be Banach spaces, and let $T:\mathcal{X}\rightarrow \mathcal{Y}$ be linear.
\begin{enumerate}
\item[(i)] 
Suppose that $R:\mathcal{Z}\rightarrow \mathcal{X}$ is a continuous surjective linear map. Then $ \C (TR)= \C (T). $
\item[(ii)]
 Suppose that $S:\mathcal{Y}\rightarrow \mathcal{Z}$ is a continuous linear map. Then $ST$ is continuous if and only if $S(\C (T))=(0)$.
\end{enumerate}
\end{lem}
\begin{prop}\label{taj2}
Let $D$ be a derivation on $\A\ltimes \U$ such that 
$$D((a,x))=(\delta_1 (a)+\tau_1 (x),\delta_2 (a)+\tau_2 (x))\quad\quad (a\in \A,x\in \U)$$
and $\delta_2(\A)\subseteq ann_{\U}\U$. Then under any of the following conditions, \[D_1((a,x))=(\delta_1(a)+\tau_1(x),\tau_2(x))\]
is a continuous derivation on  $\A\ltimes \U$ and $D=D_1+D_2$ where $D_2((a,x))=(0,\delta_2(a))$ is a derivation on  $\A\ltimes \U$.
\begin{enumerate}
\item[(i)]
$ann_{\A}\U =(0)$ and $\tau_2$ is continuous.
\item[(ii)]
$\A$ possesses a bounded right approximate identity, there is a surjective left $\A$-module homomorphism  
$\phi :\A\rightarrow \U$ and both $\delta_1$ and $\tau_1$ are continuous.
\item[(iii)]
 $\A$ possesses a bounded right approximate identity, there is an injective left $\A$-module homomorphism  
$\phi :\A\rightarrow \U$ and both $\tau_1$ and $\tau_2$ are continuous.
\end{enumerate}
\end{prop}
\begin{proof}
 From Corollary \ref{tak}-$(ii)$,
\[D_1((a,x))=(\delta_1(a)+\tau_1(x),\tau_2(x)) \quad \text{and} \quad
D_2((a,x))=(0,\delta_2(a))\quad\quad ((a,x))\in \A\ltimes \U) \]
are derivations on $\A\ltimes \U$.\\
$(i)$ Since $ann_{\A}\U =(0)$, as in Proposition \ref{au1}-$(i)$, it can be proved that $\tau_1$ is continuous. By continuity of $\tau_2$, $ann_{\A}\U =(0)$, $\delta_2(\A)\subseteq ann_{\U}\U$ and Theorem \ref{joda}-$(iii)$, we have 
$$\C (\delta_1)\subseteq (\C (\tau_2):\U)_{\A}=((0),\U)_{\A}=ann_{\A}\U =(0).$$
Therefore $\delta_1$ is continuous and so is $D_1$.\par 
Before proving parts $(ii)$ and $(iii)$ we show that if $\psi :\A\rightarrow \U$ is any left $\A$-module homomorphism, then it is continuous. Let $\{a_k\}$ be a sequence in $\A$ such that $a_k\rightarrow 0$. By the Cohen's factorization theorem there is some $c\in \A$ and some sequence $b_k$ in $\A$ for which $b_k\rightarrow 0$ and $a_k=b_k c \, \, (k\in \mathbb{N})$. Thus $\psi (a_k)=b_k\psi (c)\rightarrow 0$ and hence $\psi$ is continuous.
\\
$(ii)$ Let $\psi :\A\rightarrow \U$ be the map $\psi = \tau_2 o\phi - \phi o\delta_1$. By the condition $\delta_2(\A)\subseteq ann_{\U}\U$ and the properties of $\delta_1$ and $\tau_2$ it can be shown that $\psi$ is a left $\A$-module homomorphism. Therefore $\psi$ and $\phi$ are continuous and hence $\tau_2 o\phi$ is continuous. So $\C (\tau_2 o\phi)=(0)$. On the other hand by Lemma \ref{da}-$(i)$, since $\phi$ is surjective, $\C (\tau_2 o\phi)=\C (\tau_2)$. Thus $\C (\tau_2)=0$. Therefore $\tau_2$ is continuous. So by the hypothesis $D_1$ is continuous.
\\
$(iii)$
As in the previous part we put $\psi = \tau_2 o\phi - \phi o\delta_1$ which is a left $\A$-module homomorphism. Since $\tau_2$ and $\phi$ are continuous, it follows that $\phi o\delta_1 $ is continuous as well. By Lemma \ref{da}-$(ii)$ we have $\phi(\C (\delta_1))=(0)$. Since $\phi$ is injective, it follows that $\C (\delta_1)=(0)$. So $\delta_1$ is continuous and by the hypothesis $D_1$ is continuous as well.
\end{proof}
Note that Propositions \ref{taj2} also holds if "bounded right approximate identity" and "left $\A$-module homomorphism" are replaced respectively by "bounded left approximate identity" and "right $\A$-module homomorphism".
\begin{rem}\label{taj11}
In Propositions \ref{taj1} and \ref{taj2} if we assume that $ann_{\A}\U =(0)$, then as in Proposition \ref{au1}-$(i)$, the continuity of $\tau_2$ implies the continuity of $\tau_1$ or if we assume that $\U$ has a bounded approximate identity, we conclude that $\tau_1 =0$. Therefore any of the conditions $ann_{\A}\U =(0)$ or $\U$ to have a bounded approximate identity implies that $\tau_1$ is continuous and therefore so $D_1$ is continuous.
\end{rem}
\par 
In the following example we show that the derivation $D_2$ in Propositions \ref{taj1} and \ref{taj2}, can be discontinuous.
\begin{exm}
Let $\mathcal{B}$ be a Banach algebra and $d:\mathcal{B}\rightarrow \mathcal{B}$ be a discontinuous derivation.
\begin{enumerate}
\item[(i)]
Let $\A:=\mathcal{B}\times \mathcal{B}$ be the direct product of Banach algebras and let $\U :=\mathcal{B}$ becomes a Banach $\A$-bimodule with compatible actions, by the following module actions;
\[ 
(a,b)x=ax \quad  \text{and} \quad 
x(a,b)= xa\quad\quad (a,b\in \A , x\in \U).
\]
Therefore $(0)\times \mathcal{B}\subseteq ann_{\A}\U$.
Define the map $\delta_1: \A\rightarrow \A$ by $\delta_1((a,b))=(0,d(b))$. Then $\delta_1$ is a discontinuous derivation on $\A$ such that 
$$\delta_1 (\A)\subseteq (0)\times \mathcal{B}\subseteq ann_{\A}\U. $$
So the map $D_2: T(\A,\U)\rightarrow T(\A,\U)$ defined by $D_2((a,b),x)=(\delta_1((a,b)),0)$ is discontinuous.
\item[(ii)]
Assume that $\U :=\mathcal{B}\times \mathcal{B}$ as the direct product of Banach spaces which becomes a Banach algebra with thr product 
$$(x,y)(x',y')=(xx',0)\quad \quad ((x,y),(x',y')\in \U).$$
Let $\A:=\mathcal{B}$ and $\U$ is became a Banach $\A$-bimodule with the pointwise module actions which are compatible  with its algebraic operations.\par 
Now consider the map $\delta_2: \A\rightarrow \U$ defined by $\delta_2(a)=(0,d(a))$. $\delta_2$ is a discontinuous derivation such that 
$$\delta_2(\A)\subseteq (0)\times \mathcal{B}\subseteq ann_{\U}\U.$$
Hence the map $D_2: \A \ltimes \U \rightarrow \A \ltimes \U$ defined by 
$$D_2((a,(x,y)))=(0,\delta_2(a))$$
is a discontinuous derivation on $\A\ltimes \U$.
\item[(iii)]
Suppose that $\U$ is a Banach space and $T:\U\rightarrow \mathbb{C}$ is a discontinuous linear functional. 
Set $\A:=T(\mathbb{C},\mathbb{C})$ and we turn $\U$ into a Banach $\U$-bimodule by the actions below;
\[(a,b)x=ax \quad \text{and} \quad x(a,b)=xb  \quad\quad ((a,b)\in A, x\in \U).\]
Consider the Banach algebra $T(\A,\U)$ and the map $\tau_1:\U\rightarrow \A$ defined by 
$$\tau_1(x)=(0,T(x)).$$
So $\tau_1$ is an $\A$-bimodule homomorphism such that 
$$x\tau_1(x')+\tau_1 (x)x'=0\quad\quad (x,x'\in \U).$$
Thus the map $D_2((a,b),x)=(\tau_1(x),0)$ is a derivation on $T(\A,\U)$ which is discontinuous.
\end{enumerate}
\end{exm}
Note that in Proposition \ref{taj2} if we assume that $\delta_2$ is continuous, then derivation $D$ is continuous as well.

\section{The first cohomology group of $\A\ltimes\U$}
In this section we determine the first cohomology group of $\A\ltimes\U$ in some special cases. Throughout this section  for two linear spaces $\mathcal{X}$ and $\mathcal{Y}$, we shall write $ \mathcal{X} \cong \mathcal{Y}$ to indicate that the spaces are linearly isomorphic. 
\par 
From this point up to the last section we assume that every derivation $D$ on $\A\ltimes\U$ is of the form 
\[D(a,x)=(\delta_1 (a),\delta_2 (a)+\tau_2 (x))\quad\quad (a\in \A,x\in \U)\]
in which $\delta_1:\A\rightarrow\A$, $\delta_2:\A\rightarrow\U$ and $\tau_2:\U\rightarrow\U$
are derivations such that 
\[
\tau_2 (ax)=a\tau_2 (x)+\delta_1 (a)x+\delta_2 (a)x \,\, \text{and} \,\, 
\tau_2 (xa)=\tau_2 (x)a+x\delta_1 (a)+x\delta_2 (a) \quad\quad (a\in \A,x\in \U).
\]
Indeed, we consider that for every derivation $D$ on $\A\ltimes\U$ we have $\tau_1=0$, where $\tau_1$ is as in Theorem \ref{asll}. 
\par 
If for all $\delta\in Z^{1}(\A)$, $\delta_1(\A)\subseteq ann_{\A}\U$, then by Corollary \ref{tak}, the map $D$ given by $D((a,x))=(\delta(a),0)$ is a continuous derivation on $\A\ltimes\U$ and for every continuous derivation $D$ on $\A\ltimes\U$ the maps $D_1((a,x))=(0,\delta_2 (a)+\tau_2 (x))$ and $D_2((a,x))=(\delta_1(a),0)$ are derivations in $Z^{1}(\A\ltimes\U)$ and $D=D_1+D_2$. By these arguments, in this case,
\[Z^{1}(\A\ltimes\U)\cong \mathcal{W} \times Z^{1}(\A),\]
where $\mathcal{W}$ is a linear subspace consisting of all continuous derivations on $\A\ltimes\U$ of the form 
$D(a,x)=(0,\delta_2 (a)+\tau_2 (x))$, such that $\delta_2\in Z^{1}(\A,\U)$, $\tau_2\in Z^{1}(\U)$ and 
\[\tau_2 (ax)=a\tau_2 (x)+\delta_2 (a)x \,\, \text{and} \,\,\tau_2 (xa)=\tau_2 (x)a+x\delta_2 (a)\quad\quad (a\in\A,x\in\U).\]
\begin{thm}\label{11}
Let for every derivation $\delta\in Z^{1}(\A)$ we have $\delta(\A)\subseteq ann_{\A}\U$. If $H^{1}(\A,\U)=(0)$, then \[H^{1}(\A\ltimes\U)\cong \frac{Z^{1}(\A)\times [Hom_{\A}(\U)\cap Z^{1}(\U)]}{\mathcal{E}},\]
where $\mathcal{E}$ is the following linear subspace of $Z^{1}(\A)\times [Hom_{\A}(\U)\cap Z^{1}(\U)]$;

\[\mathcal{E}=\{(id_{a},r_{a}+id_{\U,x})\, \mid \, a\in\A, \, id_{\A , x}=0 \, \, (x\in\U)\}.\]
\end{thm}
\begin{proof}
Define the map \[\Phi:Z^{1}(\A)\times [Hom_{\A}(\U)\cap Z^{1}(\U)]\rightarrow H^{1}(\A\ltimes\U)\]
by $\Phi ((\delta,\tau))=[D_{\delta,\tau}]$
where $D_{\delta,\tau} ((a,x))=(\delta (a),\tau (x))$ and $[D_{\delta,\tau}]$ represents the equivalence class of $D_{\delta,\tau}$ in $H^{1}(\A\ltimes\U).$ By Corollary \ref{tak} the maps $D_1:\A\ltimes\U\rightarrow \A\ltimes\U$ and $D_2:\A\ltimes\U\rightarrow\A\ltimes\U$ given respectively by $D_1((a,x))=(0,\tau (x))$ and $D_2((a,x))=(\delta(a),0)$, are continuous derivations on $\A\ltimes\U$ and so $\Phi$ is well-defined. Clearly $\Phi$ is linear. Let $D\in Z^{1}(\A\ltimes\U)$. By the hypothesis and the discussion before the theorem, $D=D_1+D_2$ where  $D_1((a,x))=(\delta_1(a),0)$ and $D_2((a,x))=(\delta_1(a),0)$ are continuous derivations on $\A\ltimes\U$ and $\delta_1\in Z^{1}(\A)$, $\delta_2\in Z^{1}(\A,\U)$ and $\tau_2\in Z^{1}(\U)$. Since $H^{1}(\A,\U)=(0)$, there is some $x_0\in \U$ such that $\delta_2 =id_{\A,x_0}$. Define the map $\tau:\U\rightarrow \U $ by $\tau=\tau_2 - id_{\U,x_0}$. By the properties of $\tau$ and Remark \ref{inn2}, it follows that $\tau\in Hom_{\A}(\U)\cap Z^{1}(\U)$. Now we have
\[
D((a,x))-D_{\delta_1 ,\tau} ((a,x))=(0,id_{\A,x_0}(a)+id_{\U,x_0}(x))= id_{(0,x_0)}(a,x) \quad \quad ((a,x)\in\A\ltimes\U).\]
So $[D]=[D_{\delta,\tau}]$ and hence $\Phi ((\delta_1,\tau))=[D_{\delta_1 ,\tau}]=[D]$. Thus $\Phi$ is surjective.
\par
It can be easily checked that $\mathcal{E}$ is a linear subspace and if $(\delta,\tau)\in ker \Phi$, then $D_{\delta,\tau}$ is an inner derivation on $\A\ltimes\U$. So for some $(a_0,x_0)\in \A\ltimes\U$, 
\[D_{\delta,\tau}((a,x))=id_{a_0,x_0}(a,x)=(id_{a_0}(a),id_{\A,x_0}(a)+r_{a_0}(x)+id_{\U,x_0}(x)),\]
which implies that \[\delta =id_{a_0}, \quad \tau (x)=r_{a_0}(x)+id_{\U,x_0}(x)\quad \text {and} \quad id_{\A,x_0}=0 \quad \quad (a\in\A ,x\in\U).\] Hence $(\delta,\tau)\in \mathcal{E}$. Conversely, Suppose that $(\delta,\tau)\in \mathcal{E}$. Then for $a_0\in\A$ and $x_0\in\U$ we have $\delta =id_{a_0}$ and $ \tau =r_{a_0}+id_{\U,x_0}$ where $id_{\A,x_0}=0$ for all $a\in\A$. So 
\[
D_{\delta,\tau}((a,x))=(id_{a_0}(a),r_{a_0}(x)+id_{\U ,x_0}(x))
= id_{(a_0,x_0)}((a,x)).\]
Thus $D_{\delta,\tau}\in N^{1}(\A\ltimes\U)$ and hence $(\delta,\tau)\in ker\Phi$. So $\mathcal{E}=ker\Phi$. Therefore by the above arguments we have 
\[H^{1}(\A\ltimes\U)\cong \frac{Z^{1}(\A)\times [Hom_{\A}(\U)\cap Z^{1}(\U)]}{\mathcal{E}}.\]
\end{proof}
\begin{cor}
Suppose that for every $\delta\in Z^{1}(\A)$, $\delta (\A)\subseteq ann_{\A}\U$. If $H^{1}(\A,\U)=(0)$ and  $H^{1}(\A\ltimes\U)=(0)$, then $H^{1}(\A)=(0)$ and $\frac{Hom_{\A}(\U)\cap Z^{1}(\U)}{C_{\A}(\U)+I(\U)}=(0).$
\end{cor}
\begin{proof}
Let $\tau\in Hom_{\A}(\U)\cap Z^{1}(\U)$ and $\delta\in Z^{1}(\A)$. Since $\delta (\A)\subseteq ann_{\A}\U$, it follows that 
\[ (0,\tau)\,,\, (\delta,0)\in Z^{1}(\A)\times [Hom_{\A}(\U)\cap Z^{1}(\U)].\]
By the hypothesis and according to the preceding theorem, there exist some $a_0\in\A$ and $x_0\in\U$ such that $id_{\A,x_0}=0$ and \[(\delta,0)=(id_{a_0},r_{a_0}+id_{\U,x_0})\quad ,\quad (0,\tau)=(id_{a_0},r_{a_0}+id_{\U,x_0}).\]
Hence $\delta=id_{a_0}$ and so $H^{1}(\A)=(0)$. Moreover, $\tau=r_{a_0}+id_{\U,x_0}$ where $id_{a_0}=0$. Thus $\tau\in C_{\A}(\U)+I(\U)$.
 \end{proof}
 If we assume that $\A$ is a Banach algebra with $ann_{\A}\A=(0)$ and $\delta\in Z^{1}(\A)$ where $\delta\neq 0$, then $\delta(\A)\not\subseteq ann_{\A}\A =(0)$. So in this case, the condition  $\delta (\A)\subseteq ann_{\A}\A$ is not stisfied on $\A\ltimes\A$ for every derivation $\delta\in Z^{1}(\A)$ and this shows that this condition does not hold in general. In the following we give an example of Banach algebras which are satisfying in conditions of Theorem \ref{11}.
 \begin{exm}
Let $\A$ be a semisimple commutative Banach algebra. By Thomas' theorem \cite{tho}, we have $Z^{1}(\A)=(0)$. So in this case, for every $\delta\in Z^{1}(\A)$ and every Banach $\A$-bimodule $\U$, $\delta (\A)\subseteq ann_{\A}\U$ (in fact $\delta=0$). In particular for a semisimple commutative Banach algebra $\A$, if $\U$ is a Banach $\A$-bimodule such that $H^{1}(\A,\U)=(0)$ and $\A\ltimes\U$ is satisfying in conditions of Theorem \ref{11}, then
 \[H^{1}(\A\ltimes\U)\cong \frac{Hom_{\A}(\U)\cap Z^{1}(\U)}{R_{\A}(\U)+ I(\U)}.\]
 \end{exm}
 Note that if $\A$ is a commutative Banach algebra with $H^{1}(\A)=(0)$ and $H^{1}(\A,\U)=(0)$, then above example holds again.
 \par 
We continue by characterizing the first cohomology group of $\A\ltimes\U$ in another case.
\par 
If for each $\delta_2\in Z^{1}(\A,\U)$, $\delta_2(\A)\subseteq ann_{\U}\U$, then by Corollary \ref{tak}, the map $D ((a,x))=(0,\delta_2(a))$ is a continuous derivation on $\A\ltimes\U$ and for each derivation $D\in Z^{1}(\A\ltimes\U)$ we conclude that the maps $D_1 ((a,x))=(\delta_1 (a),\tau_2(x))$ and $D_2((a,x))=(0,\delta_2(a))$ are derivations in $Z^{1}(\A\ltimes\U)$ and $D=D_1+D_2$. So by this argument in this case,
 \[Z^{1}(\A\ltimes\U)\cong Z^{1}(\A ,\U)\times \mathcal{W}\]
 where $\mathcal{W}$ is the linear subspace of all continuous derivations on $\A\ltimes\U$ which are of the  form $D_1((a,x))=(\delta_1 (a),\tau_2(x))$ with $\delta_1\in Z^{1}(\A)\,,\tau_2\in Z^{1}(\U)$ and 
 \[\tau_2 (ax)=a\tau_2 (x)+\delta_1 (a)x\quad \text{and} \quad \tau_2 (xa)=\tau_2 (x)a+x\delta_1 (a)\quad\quad ((a,x)\in \A\ltimes\U).\]
  \begin{thm}\label{22}
 Suppose that for each derivation $\delta\in Z^{1}(\A,\U)$ we have $\delta(\A)\subseteq ann_{\U}\U$. If $H^{1}(\A)=(0)$, then 
 \[H^{1}(\A\ltimes\U)\cong \frac{Z^{1}(\A,\U)\times [Hom_{\A}(\U)\cap Z^{1}(\U)]}{\mathcal{F}},\]
 where $\mathcal{F}$ is the following linear subspace of $Z^{1}(\A,\U)\times [Hom_{\A}(\U)\cap Z^{1}(\U)]$;
 \[\mathcal{F}=\{ (id_{\A,x},r_{a}+id_{\U, x})\, \mid \, x\in\U, \, id_{a}=0 \, \, (a\in\A) \}.\]
 \end{thm}
 \begin{proof}
 Define the map 
 \[\Phi: Z^{1}(\A,\U)\times [Hom_{\A}(\U)\cap Z^{1}(\U)]\rightarrow H^{1}(\A\ltimes\U)\]
 by $\Phi ((\delta,\tau))=[D_{\delta,\tau}]$
 where $D_{\delta,\tau}((a,x))=(0,\delta(a)+\tau(x))$
 is a continuous derivation on $\A\ltimes\U$. Clearly $\Phi$ is a well-defined linear map. If 
 $D\in Z^{1}(\A\ltimes\U)$, then $D=D_1+D_2$ where $D_1((a,x))=(\delta_1 (a),\tau_2(x))$ and $D_2((a,x))=(0,\delta_2(a))$ are continuous derivations on $\A\ltimes\U$ with $\delta_1\in Z^{1}(\A), \delta_2\in Z^{1}(\A,\U)$ and $\tau\in Z^{1}(\U)$. Since $H^{1}(\A)=(0)$, there is some $a_0\in \A$ such that $\delta_1=id_{a_0}$. Consider the map $\tau=\tau_2-r_{a_0}$ on $\U$. We have  
 $\tau\in Hom_{\A}(\U)\cap Z^{1}(\U)$. Also
 \[D((a,x))-D_{\delta_2 ,\tau}((a,x))=id_{(a_0,0)}((a,x)).\]
 So $[D]=[D_{\delta_2 ,\tau}]$ and hence $\Phi ((\delta_2,\tau))=[D_{\delta_2,\tau}]=[D]$.
 Therefore $\Phi$ is surjective. A straightforward proceeding shows that $\mathcal{F}$ is a vector subspace and $ker\Phi =\mathcal{F}$. This establishes the desired vector space isomorphism.  
 \end{proof}
  \begin{cor}\label{222}
  Suppose that for any derivation $\delta\in Z^{1}(\A,\U)$, $\delta(\A)\subseteq ann_{\U}\U$. If 
  $H^{1}(\A)=(0)$ and $H^{1}(\A\ltimes\U)=(0)$, then $H^{1}(\A,\U)=(0)$ and 
  $\frac{Hom_{\A}(\U)\cap Z^{1}(\U)}{C_{\A}(\U)+I(\U)}=(0).$
  \end{cor}
  \begin{proof}
  Let $\delta\in Z^{1}(\A,\U)$ and $\tau\in Hom_{\A}(\U)\cap Z^{1}(\U)$. By the hypothesis, 
$(\delta,0),(0,\tau)\in  Z^{1}(\A,\U)\times Hom_{\A}(\U)\cap Z^{1}(\U)$. Again by the hypothesis and the preceding theorem, there are $x_0\in\U$ and $a_0\in \A$ such that $id_{a_0}=0$ and 
\[(\delta,0)=(id_{\A,x_0},r_{a_0}+id_{\U,x_0})\quad , \quad (0,\tau)=(id_{\A,x_0},r_{a_0}+id_{\U,x_0}).\]
Now the result follows from these equalities.
  \end{proof}
If we assume that $\A$ is a Banach algebra with $ann_{\A}\A =(0)$ and $\delta\in Z^{1}(\A)$ where $\delta\neq 0$, then by putting $\U =\A$ we have $\delta(\A)\not\subseteq ann_{\A}\A =ann_{\A}\U =(0)$. So in this case, the condition $\delta(\A)\subseteq ann_{\A}\U$ is not satisfied on $\A\ltimes\U$ for every derivation $\delta\in Z^{1}(\A,\U)$. This shows that this condition does not hold in general. In the following we give an example of Banach algebras which are satisfying in conditions of Theorem \ref{22}.
\begin{exm}
If $\A$ is a Banach algebra and $\U$ is a commutative Banach $\A$-bimodule such that $H^{1}(\A,\U)=(0)$, then $Z^{1}(\A,\U)=(0)$. So in this case, for every $\delta\in Z^{1}(\A,\U)$, $\delta(\A)\subseteq ann_{\U}\U$ (in fact $\delta =0$). In this case if $H^{1}(\A)=(0)$ and $\A\ltimes\U$ satisfies in Theorem \ref{22}, then
  \[ H^{1}(\A\ltimes\U)\cong \frac{Hom_{\A}(\U)\cap Z^{1}(\U)}{C_{\A}(\U)\cap N^{1}(\U)}.\]
\end{exm}
 Note that if $\A$ is a super amenable Banach algebra and $\U$ is a commutative Banach $\A$-bimodule, then above example holds.
 \par 
 In the continuation we consider a case on $\A\ltimes\U$ in which for any derivation $\delta_1\in Z^{1}(\A)$ and $\delta_2\in Z^{1}(\A,\U)$ we have $\delta_1(\A)\subseteq ann_{\A}\U$ and $\delta_2(\A)\subseteq ann_{\U}\U$.  In fact by these conditions on $\A\ltimes\U$, for every $\delta_1\in Z^{1}(\A)$ and $\delta_2\in Z^{1}(\A,\U);$
  \[\delta_1(a)x+\delta_2(a)x=0\quad \text {and} \quad x\delta_1 (a)+x\delta_2 (a)=0 \quad\quad (a\in \A,x\in \U).\]
In this case, every derivation $D\in Z^{1}(A\ltimes\U)$ can be written as $D=D_1+D_2$ where 
  \[D_{1}((a,x))=(\delta_1 (a),\delta_2(a))\quad \text{and} \quad D_{2}((a,x))=(0,\tau_2(x))\]
are continuous derivations on $\A\ltimes\U$ and $\tau\in Hom_{\A}(\U)\cap Z^{1}(\U)$. In this case, we conclude that 
$$R_{\A}(\U)+N^{1}(\U)\subseteq Hom_{\A}(\U)\cap Z^{1}(\U).$$  
\begin{thm}\label{33}
 Suppose that for every $\delta_1\in Z^{1}(\A)$ and $\delta_2\in Z^{1}(\A,\U)$ we have \[\delta_1(\A)\subseteq ann_{\A}\U \quad \text{and} \quad \delta_{2}(\A)\subseteq ann _{\U}\U.\]
Suppose further that $\frac{Hom_{\A}(\U)\cap Z^{1}(\U)}{R_{\A}(\U)+ N^{1}(\U)}=(0)$. Then 
 \[H^{1}(\A\ltimes\U)\cong \frac{Z^{1}(\A)\times Z^{1}(\A,\U)}{\mathcal{K}},\]
 where $\mathcal{K}$ is the following linear subspace of $Z^{1}(\A)\times Z^{1}(\A,\U)$;
 \[\mathcal{K}=\{(id_{a},id_{\A,x})\, \mid \,  r_{a}+id_{\U,x}=0\,\,(a\in \A,x\in \U)\}.\]
 \end{thm}
  \begin{proof}
 Define the map 
 \[\Phi: Z^{1}(\A)\times Z^{1}(\A,\U)\rightarrow H^{1}(\A\ltimes\U)\]
 by $\Phi ((\delta_1,\delta_2))=[D_{\delta_1 , \delta_2}]$
 where $D_{\delta_1 , \delta_2} ((a,x))=(\delta_1 (a),\delta_2(a))$ is a continuous derivation on $\A\ltimes\U$. The map $\Phi$ is well-defined and linear. If $D\in Z^{1}(\A\ltimes\U)$, then $D((a,x))=(\delta_1 (a),\delta_2(a)+\tau_2(x))$. By the hypothesis, $\tau_2=r_{a}+id_{\U,x}$ where $a\in\A$ and $x\in\U$. Define the derivations $d_1\in Z^{1}(\A)$ and $d_2\in Z^{1}(\A,\U)$ by $d_1=\delta_1 -id_{a}$ and $d_2=\delta_2 -id_{\A,x}$ respectively. Then $D-D_{d_1 ,d_2}=id_{(a ,x )}$. So $\Phi ((\delta_1,\delta_2))=[D_{d_1 , d_2}]=[D]$.
Thus $\Phi$ is surjective. It can be easily seen that $ker \Phi =\mathcal{K}$. So we have the desired vector spaces isomorphism.
 \end{proof}
 \begin{cor}\label{333}
 Suppose that for every $\delta_1\in Z^{1}(\A)$ and $\delta_2\in Z^{1}(\A,\U)$ we have \[\delta_1(\A)\subseteq ann_{\A}\U \quad \text{and} \quad \delta_2(\A)\subseteq ann _{\U}\U.\] Suppose further that 
  $\frac{Hom_{\A}(\U)\cap Z^{1}(\U)}{{R_{\A}(\U)+ N^{1}(\U)}}=(0)$. If $H^{1}(\A\ltimes\U)=(0)$, then $H^{1}(\A)=(0)$ and $H^{1}(\A,\U)=(0)$.
  \end{cor}
  \begin{proof}
  Let $\delta_1\in Z^{1}(\A)$ and $\delta_2\in Z^{1}(\A,\U)$. By the hypothesis 
  \[(\delta_1 , 0)\,,\, (0,\delta_2)\in Z^{1}(\A)\times Z^{1}(\A,\U).\]
 Again by the hypothesis and the preceding theorem, there exists some $(id_{a},id_{\A,x})\in \mathcal{K}$ such that 
  $(\delta_1 , 0)=(id_{a},id_{\A,x})$ and $(0,\delta_2)=(id_{a},id_{\A,x})$ and so $\delta_1$ and $\delta_2$ are inner.
  \end{proof} 
If $H^{1}(\U)=(0)$, since $R_{\A}(\U)+ N^{1}(\U)\subseteq Z^{1}(\U)= N^{1}(\U)$, It follows that 
 $\frac{Hom_{\A}(\U)\cap Z^{1}(\U)}{{R_{\A}(\U)+ N^{1}(\U)}}=(0)$. So if $H^{1}(\U)=(0)$, Theorem \ref{33} and Corollary \ref{333} hold again.  
 \par 
 Let $\A$ be a Banach algebra with $ann_{\A}\A =(0)$ and $\delta_1 ,\delta_2\in Z^{1}(\A)$ such that $\delta_1 +\delta_2 \neq 0$. If we put $\tau =\delta_1+\delta_2$ and define linear map $D$ on $\A\ltimes\A$ by 
 \[D((a,x))=(\delta_1 (a),\delta_2(a)+\tau(x))\quad\quad ((a,x)\in \A\ltimes\A),\]
 then $D\in Z^{1}(\A\ltimes\A )$. But for $a, x\in\A$, the equation $\delta_1 (a)x+\delta_2 (a)x=0$
 is not necessarily true. This example does not satisfy the conditions of Theorem \ref{33}. 
 \par 
The last case we consider is as follows.
 \begin{thm}\label{44}
 Let $H^{1}(\A)=(0)$ and $H^{1}(\A,\U)=(0)$. Then 
 \[H^{1}(\A\ltimes\U)\cong \frac{Hom_{\A}(\U)\cap Z^{1}(\U)}{C_{\A}(\U)+I(\U)}.\]
\end{thm} 
\begin{proof}
Define the map 
\[\Phi : Hom_{\A}(\U)\cap Z^{1}(\U)\rightarrow H^{1}(\A\ltimes\U)\]
by $\Phi(\tau)=[D_{\tau}]$ where $D_{\tau}((a,x))=(0,\tau (x))$ is a continuous derivation on $\A\ltimes\U$. The map $\Phi$ is well-defined linear. If $D\in Z^{1}(\A\ltimes\U )$, then $D=(\delta_1 , \delta_2 +\tau_2)$. By hypotheses  $\delta_1 =id _{a}$ and $\delta_2 =id_{\A, x}$ for some  $a\in \A$ and $x\in \U$. Define the map $\tau:\U\rightarrow \U$ by $\tau=\tau_2-r_{a}-id_{\U, x}$. Then $\tau\in Hom_{\A}(\U)\cap Z^{1}(\U)$ and so $D-D_{\tau}=id_{(a , x)}$. Hence $\Phi (\tau)=[D_\tau]=[D]$. Thus $\Phi$ is surjective. By Corollary \ref{tak}-$(iv)$, $ker \Phi =C_{\A}(\U)+I(\U)$. So the desired vector space isomorphism is established.
\end{proof}
The following corollary follows immediately from the preceding theorem.
\begin{cor}
If $H^{1}(\A)=(0),H^{1}(\A,\U)=(0)$ and $H^{1}(\A\ltimes\U)=(0)$, then for every $\tau\in Hom_{\A}(\U)\cap Z^{1}(\U)$ there are some $r_{a}\in C_{\A}(\U)$ and $id_{\U , x}\in I(\U)$ such that $\tau =r_{a}+id_{x}$.
\end{cor}
In the following an example of Banach algebras satisfying the conditions of Theorem \ref{44}, is given.
\begin{exm}
If $\A$ is a weakly amenable commutative Banach algebra, then for any commutative Banach $\A$-bimodule $\mathcal{X}$ we have $H^{1}(\A,\mathcal{X})=(0)$. So in this case, if $\U$ is a commutative Banach $\A$-bimodule satisfying the conditions of Theorem \ref{44}, then $Z^{1}(\A)=H^{1}(\A)=(0)$ and $Z^{1}(\A,\U)=H^{1}(\A,\U)=(0)$ and hence 
\[H^{1}(\A\ltimes\U)\cong \frac{Hom_{\A}(\U)\cap Z^{1}(\U)}{R_{\A}(\U)+ N^{1}(\U)}.\]
\end{exm}
This example could be also derived from Theorem \ref{22}.
\section{Applications}
In this section we investigate applications of the previous sections and give some examples.
\subsection*{Direct product of Banach algebras}
Let $\A$ and $\U$ be Banach algebras. With trivial module actions $\A\U=\U\A =(0)$, as we saw in Example \ref{dp}, $\A\ltimes\U=\A\times\U$ where $\A\times\U$ is $l^{1}$-direct product of Banach algebras. In this case, $ann_{\A}\U =\A$ and so for every derivation $\delta:\A\rightarrow \A$ we have $\delta(\A)\subseteq ann_{\A}\U$. Also in this case, $R_{\A}(\U)=(0)$ and $Hom_{\A}(\U)=\mathbb{B}(\U)$. The following proposition follows from Theorem \ref{asll}.
\begin{prop}\label{dpd}
Let $\A$ and $\U$ be Banach algebras and $D:\A\times\U\rightarrow\A\times\U$ be a map. The following are equivalent.
\begin{enumerate}
\item[(i)]
$D$ is a derivation.
\item[(ii)]
\[D((a,x))=(\delta_1(a)+\tau_1(x),\delta_2(a)+\tau_2(x))\quad \quad ((a,x)\in\A\times\U)\]
such that $\delta_1:\A\rightarrow\A,\tau_2:\U\rightarrow\U$ are derivations and $\tau_1:\U\rightarrow\A$ and $\delta_2:\A\rightarrow\U$ are linear maps satisfying the following conditions;
\[\tau_1(\U)\subseteq ann_{\A}\A , \,\, \delta_2(\A)\subseteq ann_{\U}\U, \,\,  \tau_1(xy)=0, \, \text{and} \,\, \delta_2(ab)=0 \quad \quad(a,b\in\A , x,y\in \U).\]
\end{enumerate}
Moreover, if $ann_{\U}\U=(0)$ or $\A^{2}=\A$ and $ann_{\A}\A=(0)$ or $\U ^{2}=\U$, then $\delta_2=0$ and $\tau_1=0$, respectively.
\end{prop}
By this proposition it is clear if $\A$ or $\U$ has a bounded approximate identity, then $\delta_2=0$ and $\tau_1=0$. Hence in this case every derivation $D$ on $\A\times\U$ is of the form $D((a,x))=(\delta(a),\tau(x))$ where $\delta$ and $\tau$ are derivations on $\A$ and $\U$, respectively.
\begin{rem}\label{d1}
If $\A$ and $\U$ are Banach algebras, then for any derivation $\delta:\A\rightarrow\A$ and $\tau:\U\rightarrow\U$ the maps $D_1$ and $D_2$ on $\A\times\U$ given by 
\[D_{1}((a,x))=(\delta (a),0)\quad \text{and} \quad D_2((a,x))=(0,\tau (x)),\]
are derivations. According to this point, if every derivation on $\A\times\U$ is continuous, then every derivation on $\A$ and every derivation on $\U$ is continuous.
\par 
Also let  $\A$ or $\U$ have a bounded approximate identity. So by above arguments if every derivation on $\A$ and every derivation on $\U$ is continuous, then every derivation on $\A\times\U$ is continuous.
\end{rem}
\begin{rem}\label{d2}
Let $\A$ and $\U$ be Banach algebras and $D\in Z^{1}(\A\ltimes\U)$. If $ann_{\U}\U =(0)$ or $\overline{\A^{2}}=\A$ and $ann_{\A}\A =(0)$ or $\overline{\U ^{2}}=\U$, then in the representation of  $D$ as in the preceding proposition, $\delta_2=0$ and $\tau_1=0$. So $D((a,x))=(\delta(a),\tau(x))$ where $\delta \in Z^{1}(\A)$ and $\tau\in Z^{1}(\U)$. In this case we can conclude
\[ Z^{1}(\A\times\U)\cong Z^{1}(\A)\times Z^{1}(\U) \quad \text{and} \quad N^{1}(\A\times\U)\cong N^{1}(\A)\times N^{1}(\U).\]
Moreover
\[ H^{1}(\A\times\U)\cong H^{1}(\A)\times H^{1}(\U).\]
In particular if $\A$ or $\U$ have a bounded approximate identity, then above observation holds.
\par 
In the case of isomorphism of the first cohomology group by Theorem \ref{11}, one can weaken the conditions. Indeed $\A\times\U$ with $\overline{\U ^{2}}=\U$ or $ann_{\A}\A =(0)$ satisfies the conditions of Theorem \ref{11}. Since $Hom_{\A}(\U)=\mathbb{B}(\U)$ and $R_{\A}(\U)=(0)$, in this case  
\[H^{1}(\A\times\U)\cong H^{1}(\A)\times H^{1}(\U)\]
(In fact in this case, $\mathcal{E}=N^{1}(\A)\times N^{1}(\U)$).
\par 
If $\overline{\A ^{2}}=\A$ or $ann_{\U}\U =(0)$, by symmetry as above we obtain the same result again.
\end{rem}
As the next example confirms, it is not necessarily true that for any derivation on $\A\times\U$, $\tau_1 =0$ or $\delta_1 =0$ in the decomposition of it. The example is given from \cite{ess}.
 \begin{exm}
 Let $\A$ be a Banach algebra with $ann_{\A}\A\neq 0$. Put $\U:=ann_{\A}\A$. Then $\U$ is a closed subalgebra of $\A$. Define the map $D$ on $\A\times\U$ by $D((a,x))=(x,0)$. Then $D$ is a derivation on $\A\times\U$ such that in its representation the map $\tau_1:\U\rightarrow\A$ is given by $\tau_1(x)=x$ with $\tau_1\neq 0$. 
 \end{exm}
Let $\A$ and $\U$ be Banach algebras and $\alpha:\A\rightarrow \U$ be a continuous algebra homomorphism with $\Vert \alpha\Vert \leq 1$. Then the following module actions turn $\U$ into a Banach $\A$-bimodule with the compatible actions and norm;
\[
ax=\alpha(a)x \quad \text{and} \quad
xa=x\alpha(a)\quad\quad (a\in \A,x\in \U).\]
In this case we can consider $\A\ltimes \U$ with the multiplication given by 
$$(a,x)(b,y)=(ab,\alpha(a) y+x\alpha(b)+xy).$$
We denote this Banach algebra by $\A\ltimes_{\alpha} \U$ which is introduced in \cite{bh}. In the next proposition we see that $\A\ltimes_{\alpha} \U$ is isomorphic as a Banach algebra to the direct product $\A\times \U$.
\begin{prop}\label{dal}
Let $\A$, $\U$ be Banach algebras and $\alpha:\A\rightarrow \U$ be a continuous algebra homomorphism with $\Vert \alpha\Vert \leq 1$. Consider the Banach algebra $\A\ltimes_{\alpha} \U$ as above. Then $\A\ltimes_{\alpha} \U$ is isomorphic as a Banach algebra to  $\A\times \U$.
\end{prop}
\begin{proof}
Define $\theta:\A\times \U\rightarrow\A\ltimes_{\alpha} \U$ by $\theta((a,x))=(a, x-\alpha(a))$  $(a,x)\in \A\times \U$. The map $\theta$ is linear and continuous. Also
\[\theta((a,x)(b,y))=(a,x-\alpha(a))(b,y-\alpha(b))=(ab,xy-\alpha(ab))=\theta((a,x))\theta((b,y)),\]
for any $(a,x),(b,y)\in \A\times \U$. It is obvious that $\theta$ is bijective. Hence $\theta$ is a continuous algebra isomorphism.
\end{proof}
\begin{rem}
Let $\A$, $\U$ and $\alpha$ are as Proposition \ref{dal}. By Remarks \ref{d1}, \ref{d2} and Proposition \ref{dal} we have the following.
\par 
If every derivation on $\A\ltimes_{\alpha}\U$ is continuous, then every derivation on $\A$ and every derivation on $\U$ is continuous. If  $\A$ or $\U$ have a bounded approximate identity and every derivation on $\A$ and every derivation on $\U$ is continuous, then every derivation on $\A\ltimes_{\alpha}\U$ is continuous. Also if $\overline{\U ^{2}}=\U$ or $ann_{\A}\A =(0)$ ($\overline{\A ^{2}}=\A$ or $ann_{\U}\U =(0)$), then $H^{1}(\A\ltimes_{\alpha}\U)\cong H^{1}(\A)\times H^{1}(\U)$.
\end{rem}
If $\U$ is a Banach algebra and $\A$ is a closed subalgebra of it, then it is clear that the embedding map $i:\A\rightarrow\U$ is continuous algebra homomorphism and hence we can conmsider $\A\ltimes_{i} \U$ which is isomorphic as a Banach algebra to  $\A\times \U$ by Proposition \ref{dal}. So by above remark we have the following examples.
\begin{exm}
\begin{enumerate}
\item[(i)] If $\A$ is a semisimle Banach algebra with a bounded approximate identity, then every derivation on $\A\ltimes_{i}\A$ is continuous. 
\item[(ii)] Consider the group algebra $L^{1}(G)$ as a closed ideal in the measure algebra $M(G)$. Every derivation on $M(G)$ and every derivation on $L^{1}(G)$ is continuous and $M(G)$ is unital. So every derivation on $L^{1}(G)\ltimes_{i} M(G)$ is continuous.
\item[(iii)] Let $\mathcal{H}$ be a Hilbert space and $\mathcal{N}$ be a complete nest in $\mathcal{H}$. The associated nest algebra $Alg\mathcal{N}$ is a closed subalgebra of $\mathbb{B}(\mathcal{H})$ which is unital. By \cite{chr} every derivation on $Alg\mathcal{N}$ is  continuous. Also $\mathbb{B}(\mathcal{H})$ is a unital $C^{*}$-algebra. Hence every derivation on $Alg\mathcal{N}\ltimes_{i} Alg\mathcal{N}$ and $Alg\mathcal{N}\ltimes_{i} \mathbb{B}(\mathcal{H})$ is continuous.
\end{enumerate}
\end{exm}
\begin{exm}
\begin{enumerate}
\item[(i)] Let $\A$ be a weakly amenable commutative Banach algebra. Since $H^{1}(\A)=(0)$ and $\overline{\A ^{2}}=\A$, it follows that $H^{1}(\A\ltimes_{i}\A)=(0)$.
\item[(ii)]
Sakai showed in \cite{sa} that every continuous derivation on a $W^{*}$-algebra is inner. Every von Neumann algebra is a $W^{*}$-algebra which is unital. Let $\A$ be a von Neumann algebra on a Hilbert space $\mathcal{H}$. Hence $H^{1}(\A\ltimes_{i}\A)=(0)$ and $H^{1}(\A\ltimes_{i}\mathbb{B} (\mathcal{H}))=(0)$.
\item[(iii)]
Let $\mathcal{H}$ be a Hilbert space and $\mathcal{N}$ be a complete nest in $\mathcal{H}$. In \cite{chr}, Christensen proved that $H^{1}(Alg\mathcal{N})=(0)$. Also $\mathbb{B}(\mathcal{H})$ is a von Neumann algebra. Hence  \[ H^{1}(Alg\mathcal{N}\ltimes_{i} Alg\mathcal{N})=(0) \quad \text{ and } \quad H^{1}(Alg\mathcal{N}\ltimes_{i} \mathbb{B}(\mathcal{H}))=(0).\]
\end{enumerate}
\end{exm}
\subsection*{Module extension Banach algebras}
Let $\A$ be a Banach algebra and $\U$ be a Banach $\A$-bimodule. With trivial product $\U ^{2}=(0)$, as we saw in Example \ref{me}, $\A\ltimes\U$ is the same as module extension of $\A$ by $\U$, namely $T(\A,\U)$. In this case, $ann_{\U}\U =\U$ and so for every derivation $\delta:\A\rightarrow\U$ we have $\delta(\A)\subseteq ann_{\U}\U$. By these notes, Theorem \ref{asll} and Corollary \ref{tak}, we have the following proposition on derivations on $T(\A,\U)$.
\begin{prop}\label{ttd}
 Let $D:T(\A,\U)\rightarrow T(\A,\U)$ be a linear map such that 
 \[D((a,x))=(\delta_1(a)+\tau_1(x),\delta_2(a)+\tau_2(x))\quad\quad ((a,x)\in \A\ltimes\U).\]
 The following are equivalent.
 \begin{enumerate}
 \item[(i)]
 $D$ is a derivation. 
 \item[(ii)]
 $D=D_1+D_2$ where $D_1((a,x))=(\delta_1(a)+\tau_1(x),\tau_2(x))$ and $D_{2}((a,x))=(0,\delta_2(a))$ are derivations on $T(\A,\U)$.
 \item[(iii)]
 $\delta_1:\A\rightarrow \A$, $\delta_2:\A\rightarrow \U$ are derivations, $\tau_{2}:\U\rightarrow \U$ is a linear map such that 
 \[\tau_2 (ax)=a\tau_2 (x)+\delta_2 (a)x\quad \text{and} \quad \tau_2 (xa)=\tau_2 (x)a+x\delta_2(a)\quad \quad (a\in\A,x\in\U).\] 
 and $\tau_1:\U\rightarrow \A$ is an $\A$-bimodule homomorphism such that $x\tau_1(y)+\tau_1(x)y=0$ $ (x,y\in\U)$. \\
 Moreover, $D$ is an inner derivation if and only if $\delta_1 ,\delta_2$ are inner derivations, $\tau_1 =0$ and if $\delta_1 =ad_{a}$ and $ \delta_2 =ad_{x}$, then $\tau_2 =r_{a}$.
 \end{enumerate}
 \end{prop}
  Proposition 2.2 of \cite{med} is a consequence of this proposition. 
\begin{rem} \label{r0} 
 By Proposition \ref{ttd}, the linear map $\delta:\A\rightarrow\U$ is a derivation if and only if the linear map $D((a,x))=(0,\delta (a))$ on $T(\A,\U)$ is a derivation. So any derivation $\delta:\A\rightarrow\U$ is continuous, if every derivation on $T(\A,\U)$ is continuous.
 \end{rem} 
 \begin{prop}\label{tau}
Suppose that there are $\A$-bimodule homomorphisms $\phi:\A\rightarrow\U$ and $\psi:\U\rightarrow\A$ such that $\phi o\psi = I_{\U}$ ($I_\U$ is the identity map on $\U$). If every derivation on $T(\A,\U)$ is continuous, then every derivation on $\A$ is continuous.
\end{prop}
\begin{proof}
Let $\delta$ be a derivation on $\A$. Define the map $\tau:\U\rightarrow\U$ by $\tau =\phi o\delta o\psi$.
Then for every $a\in\A,x\in\U$,
\[\tau(ax)=a\tau(x)+\delta(a)x\quad \text{and} \quad \tau(xa)=\tau(x)a+x\delta(a).\]
So the map $D:T(\A,\U)\rightarrow T(\A,\U)$ defined by $D((a,x))=(\delta(a),\tau(x))$ is a derivation which is continuous by the hypothesis. Thus $\delta$ is continuous.
\end{proof}
In the previous proposition the assumption of existence of $\A$-bimodule homomorphisms $\phi:\A\rightarrow\U$ and $\psi:\U\rightarrow\A$ with $\phi o \psi =I_{\U}$ is equivalent to that there exists a subbimodule $\mathcal{V}$ of $\A$ such that $\A=\U\oplus \mathcal{V}$ as $\A$-bimodules direct sum. (In fact in this case, $\U$ and $\mathcal{V}$ are ideals of $\A$). 
 \par 
Since for every derivation $\delta:\A\rightarrow\U$, $\delta(\A)\subseteq ann_{\U}\U$, in the stated cases in Proposition \ref{taj2}, one can express any derivation $D$ on $T(\A,\U)$ as the sum of two derivations one of which being continuous. Also the Proposition \ref{au1}-$(i)$ holds in the case of module extension Banach algebras. 
\begin{rem}\label{r1}
If $\A$ is a Banach algebra and $\mathcal{I}$ is a closed ideal on it, then $\frac{\A}{\mathcal{I}}$ is a Banach $\A$-bimodule and so we can consider $T(\A, \frac{\A}{\mathcal{I}})$. Suppose that $\A$ possesses a bounded right (or left) approximate identity and every derivation on $\A$ and every derivation from $\A$ to $\frac{\A}{\mathcal{I}}$ is continuous. Let $D$ be a derivation on $T(\A, \frac{\A}{\mathcal{I}})$ which has a structure as in Proposition \ref{ttd}. Then $\tau_1:\frac{\A}{\mathcal{I}}\rightarrow \A$ is an $\A$-bimodule homomorphism. Since $\A$ has a bounded right(left) approximate identity, then so does $\frac{\A}{\mathcal{I}}$. Hence $\tau_1$ is continuous. Now from Proposition \ref{taj2}-$(ii)$ it follows that $D$ is continuous. Hence in this case any derivation on $T(\A, \frac{\A}{\mathcal{I}})$ is continuous.
\end{rem}
\begin{rem}\label{r2}
If $\mathcal{I}$ is a closed ideal in a Banach algebra $\A$ and $\delta:\A\rightarrow\A$ is a derivation such that $\delta(\mathcal{I})\subseteq \mathcal{I}$, then the map $\tau: \frac{\A}{\mathcal{I}}\rightarrow \frac{\A}{\mathcal{I}}$ defined by $\tau(a+\mathcal{I})=\delta(a)+\mathcal{I}$ is well-defined and linear and 
\[\tau (a(x+\mathcal{I}))=a\tau (x+\mathcal{I})+\delta (a)(x+\mathcal{I})\quad,\quad \tau((x+\mathcal{I})a)=\tau(x+\mathcal{I})a+(x+\mathcal{I})\delta(a).\]
 Therefore the map $D$ on  $T(\A, \frac{\A}{\mathcal{I}})$ defined by $D((a,x))=(\delta (a),\tau(x+\mathcal{I}))$ is a derivation. So if every derivation on  $T(\A, \frac{\A}{\mathcal{I}})$ is continuous, then every derivation $\delta:\A\rightarrow \A $ with $\delta(\mathcal{I})\subseteq \mathcal{I}$ is continuous.
\end{rem}
 In Remarks \ref{r1} and \ref{r2}, if we let $\mathcal{I}=(0)$, then we have the following corollary.
 \begin{cor}
  Let $\A$ be a Banach algebra.
\begin{enumerate}
\item[(i)] If $\A$ has a right (left) approximate identity and every derivation on $\A$ is continuous, then any derivation on $T(\A, \A)$ is continuous. 
\item[(ii)] If every derivation on $T(\A, \A)$ is continuous, then any derivation on $\A$ is continuous.
\end{enumerate} 
 \end{cor}
The part (ii) of this corollary could also be derived from Remark \ref{r0} or Proposition \ref{tau}.
\par 
In continue we give some results of Proposition \ref{taj2} in the case of module extension Banach algebras.
\begin{cor}\label{trs}
Let $\A$ be a semisimple Banach algebra which has a bounded approximate identity and $\U$ be a Banach $\A$-bimodule with $ann_{\A}\U=(0)$. Suppose that there exists a surjective left $\A$-module homomorphism $\phi:\A\rightarrow\U$ and every derivation from $\A$ into $\U$ is continuous, then every derivation on $T(\A,\U)$ is continuous.
\end{cor}
\begin{proof}
Let $D$ be a derivation on $T(\A,\U)$ which has a structure as in Proposition \ref{ttd}. Since $\A$ is semisimple, every derivation from $\A$ into $\A$ is continuous. Now by Proposition \ref{taj2}-$(ii)$ and Remark \ref{taj11}, it follows that every derivation on $T(\A,\U)$ is continuous.
\end{proof}
Ringrose in \cite{rin} proved that every derivation from a $C^{*}$-algebra into a Banach bimodule is continuous. So we have the following example which satisfies the conditions of the above corrolary.
\begin{exm}
Let $\A$ be a $C^{*}$-algebra and $\U$ be a Banach $\A$-bimodule with $ann_{\A}\U=(0)$. Suppose that there exists a surjective left $\A$-module homomorphism $\phi:\A\rightarrow\U$. Hence $\A$ is a semisimple Banach algebra with a bounded approximate identity. So by \cite{rin} and Corrolary \ref{trs}, any derivation on $T(\A,\U)$ is continuous. 
\end{exm}
An element $p$ in an algebra $\A$ is called an \textit{idempotent} if $p^{2}=p$.
\begin{cor}
Let $\A$ be a prime Banach algebra with a non-trivial idempotent
$p$ (i.e. $p \neq 0$) such that $\A p$ is finite dimensional. Then every derivation
on $\A$ is continuous.
\end{cor}
\begin{proof}
Let $\U:=\A p$. Then $\U$ is a closed left ideal in $\A$. By the following make $\U$ into a Banach
$\A$-bimodule:
\[ xa=0 \quad \quad (x\in \A p, a\in \A), \]
and the left multiplication is the usual multiplication of $\A$. So we can consider $T(\A,\U)$ in this case. Since $\A$ is prime, it follows that $ann_{\A}\U=(0)$. Let $\delta:\A \rightarrow \A$ be a derivation. Define the map $\tau:\U \rightarrow \U$ by $\tau(ap)=\delta(ap)p$ $(a\in \A)$. The map $\tau$ is well-defined and linear. Also
\[\tau (ax)=a\tau (x)+\delta (a)x\quad \text{and} \quad \tau (xa)=\tau (x)a+x\delta(a)\quad \quad (a\in\A,x\in\U).\] 
Since $\U$ is finite dimensional, it follows that $\tau$ is continuous. By Proposition \ref{ttd}, the mapping $D:T(\A,\U)\rightarrow T(\A,\U)$ defined by $D((a,x))=(\delta(a),\tau(x))$ is a derivation. Now for $D$ the conditions of Proposition \ref{taj2}-$(i)$ hold and hence $D$ is continuous. Therefore $\delta$ is continuous.
\end{proof}
 Now we investigate the first cohomology group of $T(\A ,\U)$. 
 \par 
 In module extension $T(\A ,\U)$ since $\U ^{2}=(0)$, for every derivation $\delta:\A\rightarrow\U$ on $T(\A ,\U)$ we have $\delta (\A)\subseteq ann_{\U}\U$ and also $Z^{1}(\U)=\mathbb{B}(\U)$ and $N^{1}(\U)=(0)$, we may conclude the following proposition from Theorem \ref{22}. 
 \begin{prop}\label{cte}
Consider the module extension $T(\A,\U)$ of a Banach algebra $\A$ and a Banach $\A$-bimodule $\U$. Suppose that $H^{1}(\A)=(0)$ and the only continuous $\A$-bimodule homomorphism $T:\U\rightarrow\A$ satisfying $T(x)y+xT(y)=0$ is $T=0$. Then 
\[H^{1}(T(\A ,\U))\cong H^{1}(\A ,\U)\times \frac{Hom_{\A}(\U)}{C_{\A}(\U)}.\]
 \end{prop} 
In fact with the assumptions of the previous proposition, in Theorem \ref{22} we have $\mathcal{F}=N^{1}(\A)\times C_{\A}(\U)$. Proposition \ref{cte} is the same as Theorem 2.5 of \cite{med}. So it can be said that Theorem \ref{22} is a generalization of Theorem 2.5 in \cite{med}. 
\begin{rem}
 \begin{enumerate}
 \item[(i)]
 If $\U$ is a closed ideal of a Banach algebra $\A$ such that $\overline{\U^{2}}=\U$, then for any continuous $\A$-bimodule homomorphism $T:\U\rightarrow\A$ with  $T(x)y+xT(y)=0\, (x,y\in\U)$ we have $T(xy)=0$. Since $\overline{\U^{2}}=\U$ and $T$ is continuous, it follows that $T=0$. Hence in this case, for any $\delta\in Z^{1}(T(\A ,\U))$ in its representation, $\tau_1=0$. Also in this case, for any $r_a \in C_{\A}(\U)$, since $a\in Z(\A)$, then $r_{a}=0$. Therefore if $H^{1}(\A)=(0)$, in this case by Proposition \ref{22} we have 
 \[H^{1}(T(\A ,\U))\cong H^{1}(\A,\U)\times Hom_{\A}(\U).\]
 \item[(ii)]
 If Banach algebra $\A$ has no nonzero nilpotent elements, $\U$ is a Banach $\A$-bimodule and $T:\U\rightarrow\A$ is an $\A$-bimodule homomorphism such that  $T(x)y+xT(y)=0\, (x,y\in\U)$, then $2T(x)T(y)=0\, (x,y\in \U)$. Thus for any $x\in \U$, $T(x)^{2}=0$ and by the hypothesis $T(x)=0$. So in this case for any derivation on $T(\A ,\U)$, in its structure by Proposition \ref{ttd} we have $\tau_1 =0$.
\end{enumerate}
 \end{rem} 
 \begin{exm}
 If $\A$ is a weakly amenable commutative Banach algebra and $\U$ is a commutative Banach $\A$-bimodule such that for any continuous $\A$-bimodule homomorphism $T:\U\rightarrow\A$ with  $T(x)y+xT(y)=0\, (x,y\in\U)$ implies that $T=0$, then $H^{1}(\A)=(0)$, $H^{1}(\A,\U)=(0)$ and $C_{\A}(\U)=(0)$ and thus 
 \[H^{1}(T(\A ,\U))\cong  Hom_{\A}(\U).\]
 \end{exm}
Various examples of the trivial extension of Banach algebras and computing their first cohomology group are given in \cite{med}. 
\par 
If $\delta\in Z^{1}(\A,\U)$, then the map $D_{\delta}:T(\A,\U)\rightarrow T(\A,\U)$ defined by $D_{\delta}((a,x))=(0,\delta(a))$ is a continuous derivation. Now we may define the linear map $\Phi :Z^{1}(\A,\U)\rightarrow H^{1}(T(\A ,\U))$ by $\Phi (\delta)=[D_{\delta}]$. By noting that $\delta$ is inner if and only if $D_{\delta}$ is inner (by Corollary \ref{tak}), it can be seen that $ker\Phi= N^{1}(T(\A ,\U))$. Thus $H^{1}(\A ,\U)$ is isomorphic to some subspace of  $H^{1}(T(\A ,\U))$. So we have the following corollary. 
 \begin{cor}
 Let $\A$ be a Banach algebra and $\U$ be a Banach $\A$-bimodule. Then there is a linear isomorphism from  
 $H^{1}(\A ,\U)$ onto a subspace of $H^{1}(T(\A ,\U))$.
 \end{cor}
 By the above corollary from $H^{1}(T(\A ,\U))=(0)$ we conclude that $H^{1}(\A ,\U)=(0)$. In particular, if $H^{1}(T(\A ,\A))=(0)$, then $H^{1}(\A)=(0)$. We can also obtain this result from the fact that any derivation $\delta:\A\rightarrow\A$ gives rise to a derivation $D:T(\A,\A)\rightarrow T(\A,\A)$ given by $D((a,x))=(\delta(a),\delta(x))$ (by Remark \ref{r2}). So if $D$ is inner, then $\delta$ is inner.
 \subsection*{$\theta$-Lau products of Banach algebras}  
In this subsection we assume that $0\neq \theta\in \Delta (\A)$ and $\U$ is a Banach algebra. By the module action given in Example \ref{la} we turn $\U$ into a Banach $\A$-bimodule with comaptible actions and norm and if it is necessary we show this module by $ \U_{\theta}$. Note that $ann_{\A}\U_{\theta} =ker \theta$. Consider $\A\ltimes\U$ and denote it by $\A\ltimes_{\theta}\U$ which is called $\theta$-Lau product. In the continuation of this section we always consider $\A\ltimes_{\theta}\U$ as just mentioned. 
\par 
The following proposition characterizes the structure of derivations on $\A\ltimes_{\theta}\U$, which is obtained from Theorem \ref{asll}.
 \begin{prop}\label{lau-der}
Let $D:\A\ltimes_{\theta} \U\rightarrow \A \ltimes_{\theta}\U$ be a map. Then the following conditions are equivalent.
\begin{enumerate}
\item[(i)] $D$ is a derivation.
\item[(ii)] 
\[D((a,x))=(\delta_1 (a)+\tau_1 (x),\delta_2 (a)+\tau_2 (x))\quad\quad (a\in \A,x\in \U)\]
such that 
\begin{enumerate}
\item[(a)]
$\delta_1 :\A\rightarrow\A, \delta_2 :\A\rightarrow\U$ are derivations such that 
\[\theta (\delta_1 (a))x+\delta_2(a)x=0\quad \text{and}\quad \theta (\delta_1 (a))x+x\delta_2(a)=0\quad (a\in\A,x\in\U).\]
\item[(b)]
$\tau_1:\U\rightarrow \A$ is an $\A$-bimodule homomorphism such that $\tau_1(xy)=0\quad (x,y\in\U)$.
\item[(c)]
$\tau_2:\U\rightarrow \U$ is a linear map such that 
\[\tau_2(xy)=\theta (\tau_1(y))x+\theta (\tau_1(x))y+x\tau_2(y)+\tau_2(x)y\quad \quad (x,y\in\U).\]
Also $D$ is inner if and only if $\tau_1 =0, \delta_2 =0, \delta_1 =id_{a}$ and $\tau_2 =id_{\U,x}$.
\end{enumerate}
\end{enumerate}
\end{prop}
By the above proposition, for a derivation $D$ on $\A\ltimes_{\theta}\U$ we have 
\[\delta_2 (\A)\subseteq Z(\U), \quad \theta (a)\tau_1 (x) =a\tau_1(x)=\tau_1(x)a\]
and so $\tau_1(\U)\subseteq Z(\A)$. Also $x\tau_1 (y)+\tau_1 (x)y =0$ for all $x,y\in \U$ if and only if $\tau_1 (\U)\subseteq ker \theta$. Additionally, $\delta_1 (\A)\subseteq ann_{\A}\U =ker\theta$ if and only if $\delta_2 (\A)\subseteq ann_{\U}\U$.
\begin{rem} \label{cla}
If $\A$ is a commutative Banach algebra, then by Thomas' theorem \cite{tho}, for any derivation $\delta:\A\rightarrow\A$, $\delta(\A)\subseteq rad (\A)$. So in this case for any derivation $D$ on $\A\ltimes_{\theta}\U$ we always have $\delta_1 (\A)\subseteq rad (\A) \subseteq ker\theta=ann_{\A}\U $ and hence $\delta_2 (\A)\subseteq ann_{\U}\U$. Also $D=D_1+D_2+D_{3}$ where $D_1((a,x))=(\delta_1(a),0)$, $D_{2}((a,x))=(0,\delta_2(a))$ and $D_3((a,x))=(\tau_1(x),\tau_2(x))$ are derivations on $\A\ltimes_{\theta}\U$. In this case by Corollary \ref{tak}-$(i)$, for every derivation $\delta:\A\rightarrow\A$ the map $D$ on $\A\ltimes_{\theta}\U$ defined by $D((a,x))=(\delta(a),0)$ is a derivation.
\end{rem}
In the following we always assume that for any derivation $D\in \A\ltimes_{\theta}\U$ we have $\tau_1 =0$ and study the derivations under this condition. In this case, $\tau_2$ is then always a derivation on $\U$. 
\par 
The established results about the automatic continuity of derivations in section 3 obviously hold as well as in this special case $\theta$-Lau product. Now we are ready to state some results concerning the automatic continuity of derivations on $\A\ltimes_{\theta}\U$.
\par 
By the definition of $\theta$-Lau product and Corollary \ref{tak}-$(iv)$, any linear map $\tau:\U\rightarrow\U$ is a derivation if and only if the linear map $D((a,x))=(0,\tau(x))$ on $\A\ltimes_{\theta}\U$ is a derivation. According to this point, the following corollary is clear.
\begin{cor}
If every derivation on $\A\ltimes_{\theta}\U$ is continuous, then every derivation on $\U$ is continuous. In particular, if every derivation on $\A\ltimes_{\theta}\A$ is continuous, then every derivation on $\A$ is continuous.
\end{cor}
If $\A$ and $\U$ are semisimple where $\U$ has a bounded approximate identity, then by Proposition \ref{semi} every derivatin on $\A\ltimes_{\theta}\U$ is continuous. In particular if $\A$ is a semisimple Banach algebra with a bounded approximate identity, then all derivations on $\A\ltimes_{\theta}\A$ are continuous. In the case of $C^{*}$-algebras we can drop the semisimplicity of $\U$. In fact by Ringrose's result \cite{rin}, Proposition \ref{lau-der} and preceding corollary we have the next corollary.
\begin{cor}
Let $\A$ be a $C^{*}$-algebra. Then every derivation on $\A\ltimes_{\theta}\U$ is continuous if and only if every derivation on $\U$ is continuous.
\end{cor}
From Proposition \ref{taj1}-$(ii)$ and Remark \ref{cla} we have the following corollary.
\begin{cor}
Let $\A$ be a commutative Banach algebra and $\U$ a semisimple Banach algebra. Then every derivation $D$ on $ \A\ltimes_{\theta}\U $ is of the form $D=D_1+D_2$ where $D_1 ((a,x))=(0,\delta_2 (a)+\tau_2(x))$ is a continuous derivation on $\A\ltimes_{\theta}\U$ and $D_1 ((a,x))=(\delta_1 (a),0)$ is a derivation on $\A\ltimes_{\theta}\U$. In particular, in this case, if every derivation on $\A$ is continuous, then every derivation on $\A\ltimes_{\theta}\U$ is continuous.
\end{cor}
Note that if $\A$ is commutative, then $\U$ is a commutative $\A$-bimodule and hence the set of all left $\A$-module homomorphisms, all right $\A$-module homomorphisms and all $\A$-module homomorphisms are the same. Now by Proposition \ref{taj2} and Remark \ref{cla} we obtain the next corollary.
\begin{cor}
Let $\A$ be a commutative Baanch algebra which has a bounded approximate identity. Then any derivation $D$ on $\A\ltimes_{\theta}\U$ is of the form $D=D_1+D_2$ where $D_1 ((a,x))=(\delta_1 (a),\tau_2(x))$ and $D_2 ((a,x))=(0,\delta_2(a))$ are derivations on $\A\ltimes_{\theta}\U$ and under any of the following conditions, $D_1$ is continuous.
\begin{enumerate}
\item[(i)]
There is a surjective $\A$-module homomorphism $\phi:\A\rightarrow\U$ and $\delta_1$ is continuous.
\item[(ii)]
There is an injective $\A$-module homomorphism $\phi:\A\rightarrow\U$ and $\tau_2$ is continuous.
\end{enumerate}
\end{cor}
In the continuation we assume that for any continuous derivation $D$ on $\A\ltimes_{\theta}\U$ we have $\tau_1 =0$. By the definition of $\A\ltimes_{\theta}\U$, in this case, $Hom_{\A}(\U)=\mathbb{B}(\U), N^{1}(\A,\U)=(0), Z^{1}(\A,\U)=H^{1}(\A,\U), R_{\A}(\U)=C_{\A}(\U)=(0)$ and $N^{1}(\U)=I(\U)$. 
\begin{rem}\label{pri}
For any derivation $\delta\in Z^{1}(\A)$, we have $\delta(\A)\subseteq ker \theta=ann_{\A}\U $, since Sinclair's theorem \cite{sinc} implies that $\delta(P)\subseteq P$ for any primitive ideal $P$ of $\A$. So in this case for any derivation $D\in Z^{1}(\A\ltimes_{\theta}\U)$ we always have $\delta_1 (\A)\subseteq ker \theta =ann_{\A}\U $ and hence $\delta_2 (\A)\subseteq ann_{\U}\U$. Also $D=D_1+D_2+D_{3}$ where $D_1((a,x))=(\delta_1(a),0)$, $D_{2}((a,x))=(0,\delta_2(a))$ and $D_3((a,x))=(\tau_1(x),\tau_2(x))$ are in $Z^{1}(\A\ltimes_{\theta}\U)$. 
\end{rem}
Note that it is not necessarily true that $\delta\in Z^{1}(\A,\U)$ implies $\delta(\A)\subseteq ann_{\U}\U$ as the next example shows this.
\begin{exm}
Assume that $G$ is a non-discrete abelian group. In \cite{br}, it has been shown that there is a nonzero continuous point derivation $d$ at a nonzero character $\theta$ on $M(G)$. Now consider $M(G)\ltimes_{\theta} \mathbb{C}$. Every derivation from $M(G)$ into $\mathbb{C}_\theta$ is a point derivation at $\theta$. It is clear that $ann_{\mathbb{C}}\mathbb{C} =(0)$. But $d\in Z^{1}(M(G),\mathbb{C}_\theta)$ is a nonzero derivation such that $d(M(G))\not\subseteq ann_{\mathbb{C}}\mathbb{C} =(0)$.
\end{exm}
\par 
Now we determine the first cohomology group of $\A\ltimes_{\theta}\U$ in some other cases.
\par 
Linear space $\mathcal{E}$ in Theorem \ref{11} in the case of $\A\ltimes_{\theta}\U$ is $\mathcal{E}=N^{1}(\A)\times N^{1}(\U)$. Therefore by theorem \ref{11} and Remark \ref{pri} we have the next proposition.
\begin{prop}\label{a1}
If $H^{1}(\A,\U)=(0)$, then 
\[H^{1}(\A\ltimes_{\theta}\U)\cong H^{1}(\A)\times H^{1}(\U).\]
\end{prop}
\begin{cor}
Let $\A$ be a Banach algebra with $H^{1}(\A,\A_{\theta})=(0)$ and $\overline{\A^{2}}=\A$. Then $H^{1}(\A \ltimes_{\theta}\A)=(0)$ if and only if $H^{1}(\A)=(0)$.
\end{cor}
\par 
Linear space $\mathcal{F}$ in Theorem \ref{22} in the case of $\A\ltimes_{\theta}\U$ is $\mathcal{F}=(0)\times N^{1}(\U)$. So we have the following proposition.
\begin{prop}
Let for any $\delta\in Z^{1}(\A,\U)$, $\delta(\A)\subseteq ann_{\U}\U$. If $H^{1}(\A)=(0)$, then 
\[H^{1}(\A\ltimes_{\theta}\U)\cong H^{1}(\A,\U)\times H^{1}(\U).\]
\end{prop}
 The linear space $\mathcal{K}$ in Theorem \ref{33} in the case of $\A\ltimes_{\theta}\U$ is $\mathcal{K}=N^{1}(\A)\times (0)$. Therefore by Theorem \ref{33} we have the next proposition.
 \begin{prop}\label{prop8}
Let for every $\delta\in Z^{1}(\A,\U)$, $ \delta(\A)\subseteq ann _{\U}\U$.
 If $H^{1}(\U)=(0)$, then 
 \[H^{1}(\A\ltimes_{\theta}\U)\cong H^{1}(\A)\times H^{1}(\A,\U).\]
\end{prop}
The following proposition follows directly from Theorem \ref{44} and the properties of $\A\ltimes_{\theta}\U$.
\begin{prop}\label{prop10}
Suppose that $H^{1}(\A)=(0)$ and $H^{1}(\A,\U)=(0)$. Then
\[H^{1}(\A\ltimes_{\theta}\U)\cong H^{1}(\U).\] 
\end{prop}
\begin{rem}\label{akh}
From Proposition \ref{prop10} it follows that if $H^{1}(\A)=(0), H^{1}(\A,\U)=(0)$ and $H^{1}(\U)=(0)$, then $H^{1}(\A\ltimes_{\theta}\U)=(0)$. Under some conditions the converse of is also true. That is, if for any $\delta\in Z^{1}(\A,\U)$ we have $\delta (\A)\subseteq ann_{\U}\U$ and $H^{1}(\A\ltimes_{\theta}\U)=(0)$, then $H^{1}(\A)=(0), H^{1}(\A,\U)=(0)$ and $H^{1}(\U)=(0)$. It does not yield from above propositions. We prove it directly as follows.  By the hypotheses, for any $\delta_1\in Z^{1}(\A)$, $\delta_2\in Z^{1}(\A,\U)$ and $\tau\in Z^{1}(\U)$ the map $D:\A\ltimes_{\theta}\U\rightarrow \A\ltimes_{\theta}\U$ given by $D((a,x))=(\delta_1(a),0)$, $D((a,x))=(0,\delta_2(a))$ or $D((a,x))=(0,\tau_2(x))$ is a continuous derivation. If $H^{1}(\A\ltimes_{\theta}\U)=(0)$, then $\delta_2=0$ and $\delta_1, \tau_2$ are inner.
\end{rem}
 \begin{exm}
Let $\A$ be a weakly amenable commutative Banach algebra and $\overline{\U^{2}}=\U$. So $Z^{1}(\A)=(0)$, $Z^{1}(\A,\U)=(0)$ and from Remark \ref{akh} we have $H^{1}(\A\ltimes_{\theta}\U)=(0)$ if and only if $H^{1}(\U)=(0)$.
 \end{exm}
 In particular if $\A$ is a weakly amenable commutative Banach algebra, above example implies that $H^{1}(\A\ltimes_{\theta}\A)=(0)$. Note that for a weakly amenable Banach algebra $\A$ the equality $\overline{\A^{2}}=\A$ always holds.

\bibliographystyle{plain}
\bibliography{paper}

\end{document}